\documentclass[twoside,12pt, leqno]{amsart}
\usepackage{amsmath,amscd,amsthm,amssymb,amsxtra,latexsym,epsfig,epic,graphics}
\usepackage[matrix,arrow,curve]{xy}
\usepackage{graphicx}
\usepackage{diagrams}
\usepackage{tikz,color}  
\def\antiddot{\mathinner{\mkern1mu\raise1pt\vbox{\kern7pt\hbox{.}}\mkern2mu
        \raise4pt\hbox{.}\mkern2mu\raise7pt\hbox{.}\mkern1mu}}

\newcommand{\JJ}{{\mathbb J}}

\newcommand{\RR}{{\mathbb R}}

\newcommand{\ann}{{\rm{ann}}}
\newcommand{\Ext}{{\rm{Ext}}}

\newcommand{\coker}{{\rm{coker}\,}}

\newcommand{\punkt}{\hspace{-.3ex}\raise.15ex\hbox to1ex{\Huge.}}

\def \fix#1 {{\hfill\break \bf (( #1 ))\hfill\break}}

\DeclareMathOperator{\Sym}{Sym}

\DeclareMathOperator{\Hom}{Hom}

\DeclareMathOperator{\depth}{depth}
\DeclareMathOperator{\length}{length}

\DeclareMathOperator{\codim}{codim}

\newcommand{\gm}{\mathfrak m}

\def\gr{{\mathfrak {gr}}}

\newtheorem{theorem}{Theorem}[section]
\newtheorem{lemma}[theorem]{Lemma}
\newtheorem{proposition}[theorem]{Proposition}
\newtheorem{corollary}[theorem]{Corollary}
\newtheorem{conjecture}[theorem]{Conjecture}

\theoremstyle{definition}

\newtheorem{remark}[theorem]{Remark}

\newtheorem{example}[theorem]{Example}

\def\P{{\mathbb P}}

\def\image{{\rm image}}

\def\fix#1{{\bf ***Fix:} #1 {\bf ***}}

\DeclareMathOperator{\rH}{{\rm H}}
\def\fC{{\mathfrak C}}
\def\Tr{{\rm Tr}}

\def\rH{{\rm H}}

\def\soc{{\rm soc\,}}
\def\jacobian{{\rm Jac}}
\def\Rbar{{\overline R}}
\def\Ibar{{\overline I}}
\def\mm{{\frak m}}
\def\RR{{\mathcal R}}
\def\Trace{{\rm Tr}}

\def\tR{{\tilde R}}
\def\tI{{\tilde I}}
\def\tJ{{\tilde J}}
\def\tK{{\tilde K}}
\def\tH{{\tilde H}}
\def\tF{{\tilde F}}

\newarrow{Iso} -----

\def\Abar{{\overline A}}
\def\Rbar{{\overline R}}
\def\Ibar{{\overline I}}
\def\Jbar{{\overline J}}
\def\Kbar{{\overline K}}
\def\abar{{\overline \alpha}}
\def\bbar{{\overline \beta}}
\def\m{{\frak m}}

\def\Rbar{{\overline R}}
\def\Sbar{{\overline S}}

\def\gr{{\rm gr}}
\def\bD{{\mathbb D}}
\def\fD{{\frak D}}

\def\lbracket{{[\kern-1.5pt[}}
\def\rbracket{{]\kern-1.5pt]}}

\def\seq#1#2{{#1_{1},\dots,#1_{#2}}}

\makeatletter
\def\Ddots{\mathinner{\mkern1mu\raise\p@
\vbox{\kern7\p@\hbox{.}}\mkern2mu
\raise4\p@\hbox{.}\mkern2mu\raise7\p@\hbox{.}\mkern1mu}}
\makeatother

\usepackage{times}
\newdimen\x \x=12pt

\usepackage{color}

\usepackage[breaklinks,bookmarksopen,bookmarksnumbered,urlcolor=blue]{hyperref}
\hypersetup{colorlinks=true,backref=true,citecolor=blue}

\author{David Eisenbud and Bernd Ulrich}

\title{Duality and Socle Generators \\
for Residual Intersections}

\begin{document}

\begin{abstract}

We prove duality results for residual intersections that  unify and complete results of  van Straten,
Huneke-Ulrich and Ulrich, and settle  conjectures of van Straten and Warmt.

Suppose that $I$ is an ideal of codimension $g$ in a Gorenstein ring,
and $J\subset I$ is an ideal with $s = g+t$ generators such that $K := J:I$ has codimension $s$.  Let $\Ibar$ be the image of $I$ in $\Rbar :=R/K$.

In the first part of the paper we prove, among other things, that under suitable hypotheses on $I$, the truncated Rees ring
$$
\Rbar\oplus \Ibar\oplus\cdots\oplus \Ibar^{t+1}
$$
is a Gorenstein ring, and that the modules $\Ibar^{u}$ and $\Ibar^{t+1-u}$ are dual
to one another via the multiplication pairing into ${\Ibar^{t+1}} \cong {\omega_{\overline R}}$.

In the second part of the paper we study the analogue of residue theory, and prove that, when $R/K$ is a finite-dimensional algebra over a field of characteristic 0 and certain other hypotheses are satified, the socle of $I^{t+1}/JI^t \cong {\omega_{R/K}}$ is generated by a Jacobian determinant.
\end{abstract}

\maketitle

\let\thefootnote\relax\footnote{
This paper reports on work begun during the Commutative Algebra Program, 2012-13,
at MSRI. We are grateful
to MSRI for providing such an exciting environment, where a chance meeting
led to the beginning of the work described here. Both authors are grateful to the
National Science Foundation for partial support. The second author was
also supported as a Fellow of the Simons Foundation.}

\noindent There are two important aspects of duality for local complete intersections. Let
$T = k[[\seq x n]]/(\seq a {n-d})$ be a power series ring over a field $k$ modulo an ideal generated by the regular sequence $\seq a {n-d}$.
The first aspect is so central that it has become a definition: such a ring $T$ is \emph{Gorenstein}---that is,
$T\cong \omega_T$, the canonical module of $T$. In the case where $T$ is 0-dimensional, this means that $T\cong \Hom_{k}(T,k)$ as a $T$-module;
and more generally that $T\cong \Hom_A(T,A)$ as $T$-module, where $A$ 
is a
Noether normalization of $T$.

The second important aspect is the theory of residues, which we think of as the explicit identification of the canonical module. Suppose that $(T, \m)$ is a reduced, equidimensional complete local $k$-algebra
of dimension $d$, where $k$ is a perfect field, and let $L$ be its total ring of quotients. Let $A = k\lbracket \seq x d\rbracket \subset T$ be a separable
Noether normalization, that is, $T$ is module finite over $A$ and  $L$ is a product of separable field extensions of $K$, the quotient field of $A$.  We think of the canonical module $\omega_{T}$ as
$\Hom_{A}(T,A)$, which, after tensoring with $K$, is generated by the trace map $\Tr_{L/K}$. Thus
there is a fractional ideal $\fC(T/A) \subset L$, called the {\it Dedekind complementary module}, such that
$$
 \Hom_A(T,A) = \fC(T/A)\; \Tr_{L/K}.
$$
 The resulting representation of $\omega_{T}$ as
 $$\fC(T/A) \, dx_{1}\wedge\cdots\wedge dx_{d} \subset  L \, dx_{1}\wedge\cdots\wedge dx_{d} =L \otimes_T \bigwedge^d \Omega_{T/k},
  $$
 where $\Omega_{T/k}$ is the universally finite module of differentials, is independent of the choice of $A$. The  usual residue map ${\rm H}^{d}_{\m}(\omega_{T}) \to k$,
  which serves to make local duality explicit, is then defined by representing
 an element $\alpha \in {\rm H}^{d}_{\m}(\omega_{T})$ as a \v Cech class
$$
\alpha = \left[
\begin{matrix}
 f dx_{1}\wedge\cdots \wedge dx_{d} \\
 x_{1}\cdot \cdots \cdot x_{d}
\end{matrix}
\right]\,,
$$
for suitable $f\in \fC(T/A)$ and suitable $A$,
and mapping $\alpha$ to $\Tr_{L/K}(f)(0)$.  For all this, see for example Kunz~\cite[Chapter 10]{K2}.

A goal of the theory is thus to compute $\fC(T/A)$.
When $T$ is a complete intersection,
the classical theory says that
$$
\fC(T/A) = \Delta^{-1}T,
$$ where $\Delta$ is the Jacobian determinant of $T$ over $A$. Equivalently,
$\Tr_{T/A} $ is $\Delta$ times a generator $\sigma$ of $\Hom_{A}(T,A) \cong T$.

\def\Tbar{{\overline T}}
These statements imply that, if $k$ has characteristic 0 and $T$ is a complete intersection, then  $\Delta\Tbar$
is the socle of $\Tbar := T/(x_{1},\dots,x_{d})$. The well-known argument goes as follows:
Since $T$ is Cohen-Macaulay, the fact that the trace is $\Delta$ times $\sigma$ is preserved if we factor out $x_{1},\dots,x_{d}$ to get a
zero-dimensional ring $\Tbar$. Since the maximal ideal $\m\Tbar$ is nilpotent, the
trace $\Tr_{\Tbar/k}$ annihilates $\m \Tbar$, but, because the characteristic is 0, the trace is not zero.
It follows that $\Tr_{\Tbar/k} =\Delta\overline \sigma$ generates the socle of $\Hom_{k}(\Tbar,k)= \overline \sigma \Tbar$.
Thus $\Delta \Tbar$ is the socle of $\Tbar$. In Section~\ref{Sclassic socle} we give the classical proof for complete intersections.

In this paper we provide analogous duality results for residual intersections. We recall the definition: Let $I$ be an ideal of codimension $g$ in a local Gorenstein ring $R$, and let $s\geq g$. A \emph{residual intersection} (or \emph{$s$-residual intersection}) of
$I$ is a proper ideal $K$ of codimension at least $s$ that can be written in the form
$K = J:I$, where $J\subset I$ is an ideal generated by $s$ elements. We set $t = s-g$. We will use this notation for the rest of this introduction.  We think of $t$ as measuring how far $J$ is from being a complete intersection. The case when $I$ is unmixed and $t=0$ is the case of linkage (\cite{PS}). The class of residual intersections contains the ideals of maximal minors of sufficiently general matrices and many other examples. Our general results have technical hypotheses, so we begin with an example.

\bigbreak
\noindent{\bf Duality.} Suppose that
$I$ is generated by a regular sequence of length $g$ in a  local Gorenstein ring $R$ with infinite residue field,
and $J$ is generated by $s = g+t< \dim R$ elements chosen generally inside
the maximal ideal times $I$. The ideal $K = J:I$ is then an $s$-residual intersection (even a geometric $s$-residual intersection, as defined in Section~\ref{section hypotheses}).
 We write $\Ibar$ for
the image  of $I$ in $\Rbar:=R/K$. By a result of Huneke and Ulrich~\cite{HU} (see Theorem ~\ref{basic 1}), the canonical module of
$\Rbar$ is $\Ibar^{t+1}$; in particular, when $t=0$, the truncated Rees algebra $\Rbar\oplus\Ibar$ is Gorenstein.
We show for arbitrary $t$ that the truncated Rees algebra
$$
\Rbar\oplus \Ibar\oplus \Ibar^{2}\oplus\cdots\oplus \Ibar^{t+1}
$$
is Gorenstein, which implies that the complementary intermediate powers $\Ibar^u$ and $\Ibar^{t+1-u}$ are dual to each other
via the multiplication pairing into $\Ibar^{t+1}$.
We also prove corresponding results for the truncated associated graded ring
$$
\Rbar/\Ibar\oplus \Ibar/\Ibar^{2}\oplus\cdots\oplus \Ibar^{t}/\Ibar^{t+1}
$$
(Theorem~\ref{associated graded} and Proposition~\ref{gor from gor}).

\medbreak
\noindent{\bf Residues.} To illustrate the second main result of this paper, again in the case where $I$ is a complete intersection,
we suppose in addition to the above that $R$ is a power series ring in $d$ variables $x_i$ over a field of characteristic 0.
Let  $A = k\lbracket x_{s+1}, \dots, x_{d}\rbracket$ be a general Noether normalization of $\Rbar$.
Write   $J = (\seq a s)$ and set
$$
\Delta = \det
\begin{pmatrix}
\frac{ \partial a_1}{\partial x_1} &\dots & \frac{ \partial a_1}{\partial x_s}\\
\vdots&\ddots&\vdots\\
\frac{ \partial a_s}{\partial x_1} &\dots & \frac{ \partial a_s}{\partial x_s}\\
\end{pmatrix}.
$$
We strengthen the statement $\omega_{\Rbar} \cong \Ibar^{t+1}$ by proving in Theorem~\ref{Dedekind complementary module} that
$$
 \fC(\Rbar/A) = \Delta^{-1}\Ibar^{t+1}
 $$
 if $\Rbar$ is reduced. This gives an explicit description of the complementary module
 of residual intersections. 
 
As an application, in Corollary~\ref{determinantal rings}, 
we give a formula for the complementary module of any reduced ring defined by an ideal of maximal minors of generic codimension.

We also apply Theorem~\ref{Dedekind complementary module} to certain 0-dimensional residual intersections, with the goal of identifying the  socles of their canonical modules as Jacobian determinants.
For example, when $R/K$ is 0-dimensional, we obtain a formula for the socle of
$I^{t+1}/JI^t \cong \omega_{R/K} \, $: it is generated by the image of an element of the form $\Delta+p$, where
$p\in(a_{1}, \dots, a_{s-1})$
(Theorem~\ref{socle is jacobian}). In general $\Delta$ itself is not even in $I^{t+1}$,
but, when it is, it generates the socle.  

We show that $\Delta\in I^{t+1}$ when the generators $a_{j}$ of $J$ are forms of the same degree and $I$ is radical (Theorem~\ref{Jacobian containment}). In Proposition \ref{socle in principal case} we prove this without the radical condition when $I$ is principal---already a nontrivial computation. In general, we do not know whether the radical condition is necessary.

When the generators of $J$ have  different degrees, the ideal $(\Delta)$ depends on the choice of generators,
and in this case $\Delta$ may not be in $I^{t+1}$ (Example~\ref{counter-example1}). We show that this can even happen when
$J$ is generated by the partial derivatives of a quasi-homogeneous polynomial, and thus have the same degrees with respect to an appropriate weighting (Example~\ref{counter-example2}).
\bigbreak

Our results are much more general than the setting above. We assume that $R$ is Gorenstein and that $I$ satisfies two sorts of conditions: one on the local numbers of generators and the other that $\depth (R/I^{u}) \geq \dim R/I-u+1$ for some range of values of $u$. We assume that
$K = J:I$ is an $s$-residual intersection of $I$ and we set
$t = s-\codim I$.

Our main results on duality are Theorems \ref{duality} and \ref{general duality}, which unify and complete a number of results of Huneke, Ulrich and van Straten.
Theorem~\ref{duality} says that
$$
\frac{I^{u}}{JI^{u-1}} \hbox{\quad is dual to\quad }\frac{I^{t+1-u}}{JI^{t-u}} \hbox{\quad for $u=0,\dots, t+1$,}
$$
where, in the case $u=0$, we interpret $JI^{-1}$ as $J:I$.
In fact we show that the duality is given in the most natural way, by multiplication,
$$
\frac{I^{u}}{JI^{u-1}}\otimes\frac{I^{t+1-u}}{JI^{t-u}} \rTo^{\hbox{\quad \small{mult} \quad}} \frac{I^{t+1}}{JI^{t}} \cong \omega_{R/K}.
$$
On the other hand,
Example~\ref{mystery module} shows that the duality statement above can hold even when the multiplication maps are not perfect pairings.

Theorem \ref{general duality} gives a deformation condition under which such dualities hold that is in many cases  more general than the condition of Theorem~\ref{duality}. In Section~\ref{codim2 examples} we present examples showing the necessity of some of the hypotheses.

In Theorems~\ref{Dedekind complementary module}, \ref{socle is jacobian}, and \ref{Jacobian containment} we  prove theorems about  $\fC(\Rbar/A)$ and the socle  extending the results described above to the more general case as well.

\vskip .5cm
\noindent{\bf History.}
Residual intersections have a long history in Algebraic Geometry, perhaps beginning with Chasles' Theorem that there are 3264 conics in the complex projective plane that are tangent to 5 general  conics \cite{C}. The theory became part of commutative algebra with the work of Artin and Nagata \cite{AN}. They asserted the Cohen-Macaulay property of residual interstections, but stated it more generally than it is true. The error was corrected by Huneke \cite{H2}, and a series of papers, culminating in \cite{U}, gave stronger and stronger results in this direction (see also \cite{H}, \cite{HN}, \cite{CNT}).

The first duality results for residual intersections were proven  by Peskine and Szpiro~\cite{PS} in the case $t=0$, the theory of liaison: if $R$ is a local Gorenstein ring and 
$J\subsetneq I$
are ideals  of the same codimension with $R/I$ Cohen-Macaulay and $J$ generated by a  regular sequence, then
$I/J$ is the canonical module of $\Rbar = R/K=R/(J:I)$. The formula for $\fC(\Rbar/A)$ in this case can be found in \cite[3.5(a)]{KW}.

For $t>0$, such results were  considered in two separate lines of work, starting about 25 years ago. In one,
Duco van Straten showed that if $J$ is one-dimensional and $t=1$, then the
module $I/J$ is self-dual. Around the same time work of Huneke and Ulrich \cite{HU},
generalizing the corresponding statement in the theory of linkage \cite{PS}, showed that,  for any $s$ and $t$, under suitable hypotheses
on $I$, the modules
$R/K$ and $I^{t+1}/JI^{t}$ are dual to one another; in particular, $I^{t+1}/JI^{t} \cong \omega_{R/K}$. The paper~\cite{CNT} gives another version of the duality, to which we will return in Section~\ref{codim2 examples}.

Comparing our Theorem~\ref{duality}, we see that the result of Huneke and Ulrich is the case $u=0$, while the result of van Straten is included in our result for $t=1$.

 Van Straten's result, cited above, appears with geometric applications in the
papers of van Straten and Warmt (\cite{SW}, \cite{W}).  Sernesi \cite{S} gives further 
geometric applications.

 \subsection*{Conjectures of van Straten and Warmt} The paper of van Straten and Warmt contains  interesting conjectures, which we were able to settle in much generalized form. The conjectures \cite[Conjecture 7.1, (1)-(3)]{SW} are essentially as follows:

\begin{conjecture}  \label{SW}
{\rm Suppose that $J$ is an ideal of codimension $g$ and dimension 1, with $s=g+1=d$ generators, in a power series ring $R= k\lbracket x_1,\dots, x_{d} \rbracket $ over a field $k$ of characteristic 0, and $I$ is the unmixed part of $J$, so that $I/J$ has finite length.
(Note that in this case van Straten's original result shows that $I/J$ is self-dual.) If $I$ is a radical ideal and $I\neq J$, then:}

\begin{enumerate}
\item The module $I/J$ is self-dual by a pairing that factors through the multiplication map $I/J \otimes I/J \to I^2/JI$.
\item The $R$-module $I^2/JI$ has a one-dimensional socle.
\item The socle of $I^2/JI$ is generated by the Jacobian determinant of the generators
of $J$.
\end{enumerate}
\end{conjecture}

Van Straten and Warmt were particularly interested in the case when $J$ is generated by the partial derivatives of a given power series $f$.

In our terms (see Section~\ref{section hypotheses}), the ideal $I$ in the conjecture satisfies the Strong Hypothesis ($G_s$ because it is reduced and the depth conditions because it is Cohen-Macaulay of dimension 1).
We give a proof of Conjecture (1) (Theorem~\ref{duality}) in a  more general setting. Conjecture (2)  was in fact
already known \cite[2.9]{U}, also in a more general setting.

As stated, Conjecture (3) is false, even for the case when  $J$ is generated by the partial derivatives of a quasi-homogeneous polynomial,
and we give a counter-example in Example~\ref{counter-example2}.
However, we prove Conjecture (3) in Theorem~\ref{Jacobian containment}, again in a more general setting,
under the additional hypothesis that $J$ is generated by
homogeneous polynomials of the same degree.

\subsection*{Acknowledgements} The results of this paper owe a great deal to the program Macaulay2 \cite{M2}, which enabled us to determine the limits of validity of many of the assertions below; some of those computations are represented by examples in the current paper. We are also grateful to Craig Huneke, whose work on residual intersections inspired and guided the whole subject.

\section{Definitions, Hypotheses and Notation}\label{section hypotheses}

Let $I$ be an ideal of codimension $g$ in a Noetherian local ring $R$. 
A proper ideal of the form
$K=J:I$ is called an $s$-residual intersection (of $I$ with respect to $J$), for some integer $s \geq g$, if
$J\subset I$ is generated by $s$ elements and
$K$ has codimension at least $s$. The ideal $K$ is said to be a \emph{geometric} $s$-residual intersection if in addition $\codim(I+K)\geq s+1.$

In order for an $s$-residual intersection of $I$ to exist, it is clearly necessary that $I$ be generated by $s$ elements locally at every prime of codimension $<s$, and for a geometric $s$-residual intersection to exist, this condition must also be satisfied at primes of codimension $s$ containing $I$. For inductive purposes, the proofs of most results in the theory require a slightly stronger hypothesis:

The ideal $I$ is said to satisfy the condition $G_{s}$ if $\mu(I_{P})\leq \codim P$ for all prime ideals $P\supset I$ with $\codim P \leq s-1$.

For example, the homogeneous ideal of any smooth variety in $\P^n$ satisfies $G_{n+1}$.

\smallbreak

The significance of the condition $G_s$ is in the following result, which allows an induction that we will use often.  \smallbreak

\begin{lemma}~\label{principal ideal}\label{general position}
Let $R$ be a Noetherian local ring with infinite residue field, and let $I\subset R$ be an ideal that satisfies $G_s$.
Let $\mathfrak a \subsetneq I$ be any ideal with $\codim (\mathfrak a :I) \geq s$. Let $a_{1},\dots, a_{s}$ be general elements of $\mathfrak a$, and set
$J_{u} = (a_{1},\dots, a_{u}),\  K_{u}=J_{u}:I$. Write $R_u= R/K_{u}$.
For $g \leq u \leq s$ the ideal
$K_{u}$ is a $u$-residual  intersection, and this residual intersection is
geometric if $u < s$.
\end{lemma}

Here, and in the rest of the paper, the notion of a set of general elements may be defined as follows. Let $R$ be a Noetherian local ring with infinite residue field $k$, and let $\mathfrak a$ be an ideal. We say that the elements $a_{1},\dots,a_{s}\subset \mathfrak a$ are general in $\mathfrak a$ if the image of the point $(a_{1}, \dots, a_{s})\in \mathfrak a^{s}$ in
$ (k\otimes_{R}\mathfrak a)^{s}
$
is general.

\begin{proof} The result follows from the theory of basic elements~\cite{EE}. For a detailed treatment, see
\cite[Section 1]{U}, and in particular \cite[1.5(ii),(iii)]{U}.
\end{proof}

Now suppose in addition that $R$ is Gorenstein. We say that $I$ satisfies the {\bf Standard Hypothesis} (respectively {\bf Weak} or {\bf Strong Hypothesis})
with respect to $s = g+t$
if $I$ satisfies $G_{s}$ and, in addition, the \emph{Depth Conditions}
$$
 \depth(R/I^{j})\geq \dim(R/I) -j+1
$$
for $j\leq t$  (respectively $j\leq t-1$ or $j\leq t+1$).
\smallbreak

For example, if
$t=1$, then the Standard Hypothesis is equivalent to the condition that $R/I$ is Cohen-Macaulay and $I$ is
generically a complete intersection. Also note that if $s = \dim R$ then the  Strong Hypothesis
is the same as the weak hypothesis, since the extra requirement is that the
depth of $R/I^{t+1}$ is $\geq 0$.
\smallbreak
Assuming that the ideal $I$ satisfies $G_{s}$, the Strong Hypothesis holds, for example,
if the Koszul homology modules $H_i(I)$ of some generating sequence of $I$ are Cohen-Macaulay in the range
$0 \leq i \leq t$ \cite[2.10]{U}; in particular it
holds for strongly Cohen-Macaulay ideals; and thus it is satisfied
by Cohen-Macaulay almost complete intersection ideals, Cohen-Macaulay ideals
of deviation 2 \cite[p.259]{AH}, and
ideals in the linkage class of a complete intersection \cite[1.11]{H1}. Standard
examples include perfect ideals of codimension 2 and perfect Gorenstein ideals of codimension 3 \cite[proof of  the only theorem]{Wa}.

The ideal of the Veronese surface in $\P^5$ satisfies the Standard hypothesis with $s=4$ and the Weak Hypothesis with $s=5$---this is the ideal that appears
in the five conics problem of Chasles~\cite{C}. (It also satisfies ``sliding depth'' for the Koszul homology, so the general residual intersection
$K:= (a_1,\dots, a_5):I$ is unmixed---see~\cite[2.3 and 3.3]{HVV}. By a Bertini argument as in the proof of Proposition~\ref{red lemma} the ideal $K$ is the homogeneous ideal of a set of reduced points.)

\section{Duality Results}\label{Sduality}

\noindent We will assume throughout this section that
$I$ is an ideal of codimension $g$ in the local Gorenstein ring $R$, and
$K=J:I$ is an $s$-residual intersection for some $s\geq g$. We set $t = s-g$. When we refer to the Standard, Weak, or Strong Hypothesis, it
will always be with respect to $s$.
\medbreak
In this section we give precise statements of our main duality results. Proofs will be found in Section~\ref{Sproofs}.

Huneke gave a simple proof of van Straten's $t=1$ result in a more general context. We include it with his gracious permission:

\begin{theorem} \label{Huneke}
Suppose that $R/I$ is Cohen-Macaulay of codimension $g$ and $J = (a_1,\dots, a_{g+1})\subset I$ is such that
$K = J:I$ has codimension $g+1$, then
the $R/K$-module $I/J$ is self-dual; that is,
$$
I/J\cong \Hom_{R}(I/J, \omega_{R/K}).
$$
\end{theorem}

\medskip

Assuming the Standard Hypothesis allows us to extend the result to higher values of $t$, and to prove a statement that is stronger even in the case $t=1$:

\begin{theorem}\label{duality} Under the Standard Hypothesis, Theorem~\ref{mu definition} applies to give an injective  map
$\mu_{t}: I^{t+1}/JI^{t}\to \omega_{R/K}$. For $1\leq u\leq t$, both the multiplication map
$$
m(I,u,t): I^{u}/JI^{u-1}\otimes I^{t+1-u}/JI^{t-u} \rTo^{\rm mult}  I^{t+1}/JI^{t}
$$
and the composition $\mu_{t}\circ m(I,u,t)$
are perfect pairings.
\end{theorem}

If in addition $I$ satisfies the Strong Hypothesis, then the duality of Theorem~\ref{duality} holds in the full range
$0 \leq u \leq t+1$.
Here, when $u=0$, we interpret $I^{u}/JI^{u-1}$ as $R/(J:I)=R/K$, and the statement is simply that
$I^{t+1}/JI^t \cong \omega_{R/K}$
and $R/K \cong \hbox{End}(\omega_{R/K})$, which follows from Theorem~\ref{basic 1}.

Note that the hypothesis of Theorem~\ref{Huneke} does not include the condition $G_s$; on the other hand, Example~\ref{mystery module} shows that
the duality asserted in Theorem~\ref{Huneke} does not necessarily come from the multiplication map as in Theorem~\ref{duality}. Examples suggest that the weaker result should also be true with a condition weaker than $G_{s}$:

\begin{conjecture}\label{optimal duality}
The duality
$$
I^{u}/JI^{u-1}\cong \Hom_{R} (I^{t+1-u}/JI^{t-u}, \omega_{R/K})
$$
holds for $1 \leq u \leq t \,$ if $K=J:I$ is an $s$-residual intersection and $I$  satisfies a weakened Standard Hypothesis with $G_{s}$ replaced by $G_{s-1}$.
\end{conjecture}

The conjecture is immediate in the case where $R$ is regular and $g=1$:
then $I = (G)$ is principal, and $J = (GF)$, where $F$ is a regular sequence (of length $s$).
In this case the pairings all reduce to the usual
isomorphisms $R/(F) \to \Hom(R/(F), R/(F))$ induced by multiplication.
We will prove the Conjecture under an additional assumption in Corollary~\ref{almost optimal} of Theorem~\ref{general duality} below.

The condition $G_{s}$ in the Strong, Standard, and Weak Hypotheses is used in the inductive proof of many theorems about residual intersections, but it is not clear why it should be necessary. Recent work (\cite{H}, \cite{HN}, \cite{CNT}) has aimed at removing this hypothesis, and has had  success in the case when $I$ is strongly Cohen-Macaulay. In particular, Chardin et al \cite{CNT} have proved an analogue of Theorem~\ref{duality} in this setting, replacing the modules
$I^{u}/JI^{u-1}$ with the modules $\Sym_{u}(I/J)$. In Section~\ref{codim2 examples} we will see that this statement does not extend too far beyond the strongly Cohen-Macaulay case; see Examples~\ref{can't drop deformation} and \ref{Johnson}.

Under the Strong Hypothesis we can
combine all the dualities of Theorem~\ref{duality}
in the statement that a certain quotient of the Rees algebra $R[Iz]$ of $I$ is Gorenstein:

\begin{corollary}\label{Rees is Gor}
Under the Strong Hypothesis, the ring
\begin{align*}
\RR &:= R/K \oplus I/J \oplus I^{2}/JI \oplus \ \ldots \ \oplus I^{t+1}/JI^{t} \\
&= R[Iz]/(K, Jz, (Iz)^{t+2})
\end{align*}
is Gorenstein.
\end{corollary}

As an application of Theorem~\ref{duality} and Corollary~\ref{Rees is Gor} we will deduce:

\begin{theorem}\label{associated graded}
In addition to the Strong Hypothesis, suppose that $K=J:I$ is a geometric $s$-residual intersection.
\begin{enumerate}
\item Let $\Ibar\subset \Rbar:=R/K$ be the image of $I$.
The truncated Rees algebra
$$
\Rbar \oplus \Ibar \oplus \Ibar^{2}\oplus\ \ldots \ \oplus \Ibar^{t+1}
$$
is Gorenstein. In particular, $\Ibar^{t+1}\cong \omega_{\Rbar}$ and the multiplication maps
$\Ibar^{u}\otimes \Ibar^{t+1-u}\to \Ibar^{t+1}$ are perfect pairings.

\item Let $I'\subset R' := R/(K+I^{t+1})$ be the image of $I$. The associated graded ring
$\gr_{I'}(R')$ is Gorenstein.
\end{enumerate}
\end{theorem}

\vspace{0.1cm}

Sometimes the duality statements of Theorem~\ref{duality} hold only for
a restricted range of values of $u$. Our most general result involves another definition: We say that
a pair $(\tR, \tI)$ consisting of a Noetherian local ring $\tR$ and an ideal $\tI$ is a \emph{deformation} of the
pair $(R,I)$ if $\tR$ contains a regular sequence $x_{1},\dots, x_{n}$ whose image in $\tR/\tI$ is also a regular sequence
such that $R \cong \tR/(x_{1}, \dots, x_{n})$ and $I = \tI R$.

\begin{theorem}\label{general duality}
Suppose that $(R,I)$ has a deformation
$(\tR,\tI)$ such that $\tI$ satisfies the condition $G_{s}$ and
the Koszul homology $H_{i}(\tI)$ is Cohen-Macaulay for $0\leq i\leq t = s-g$.
Assume further that  $I$ satisfies the condition $G_{g+v}$ for some $(t-1)/2\leq v\leq t$.

Let $\tJ$ be a lifting of
$J$ to an ideal with $s$ generators contained in $\tI$. 
The ideal
$\tK = \tJ:\tI$ is an $s$-residual intersection of $\tI$. 
Our hypothesis implies that Theorem~$\ref{mu definition}$ holds with $\tK$ in place of $K$
and gives an isomorphism $\mu_{t}$.
The inverse $\phi: \omega_{\tR/\tK}\to \tI^{t+1}/\tJ\tI^{t}$  of  $\mu_{t}$
 induces a map $\phi': \omega_{R/K}\to I^{t+1}/JI^{t}$. We have$:$

\begin{enumerate}
\item $\phi'$ is a surjection, and is an isomorphism if $K$ is a geometric $s$-residual intersection.
\item There are perfect pairings
$$
m: I^{u}/JI^{u-1}\otimes I^{t+1-u}/JI^{t-u} \rTo \omega_{R/K}
$$
for
$$
t-v \leq u \leq v+1
$$
or, equivalently, for
$$
\frac{t+1}{2} -\varepsilon \leq u \leq \frac{t+1}{2} +\varepsilon \, ,
$$
where $\varepsilon = v-(t-1)/2$.

\item If the perfect pairing $m$ is chosen as in Figure $1$ in the proof, then
$\phi'\circ m$ is the map induced by multiplication $I^{u}\otimes I^{t+1-u}\to I^{t+1}.$
\end{enumerate}
\end{theorem}

Under the hypotheses of Theorem~\ref{general duality}, the ideal $\tI$ statisfies the Strong Hypothesis (\cite[2.10]{U}). Thus
$\tR/\tK$ is Cohen-Macaulay with canonical module $\tI^{t+1}/\tJ\tI^{t}$ by Theorem~\ref{basic 1}. From the proofs below it follows that
the map $\phi'$ can also be described as a composition
$$
\omega_{R/K}\hookrightarrow R\otimes_{\tR}\omega_{\tR/\tK} \rTo^{R\otimes_{\tR}\phi} R\otimes_{\tR} \tI^{t+1}/\tJ\tI^{t} \twoheadrightarrow I^{t+1}/JI^{t}.
$$

We remark that all the hypotheses of Theorem~\ref{general duality} are satisfied when $I$ is licci and satisfies
$G_{g+v}$ (\cite[1.11]{H1} and \cite[proof of 5.3]{HU}). We will see  that the $G_{g+v}$ assumption cannot be weakened to $G_{g+v-1}$,
 even when $I$ is a codimension 2 perfect ideal (Example~\ref{can't weaken G}),
and also that the deformation assumption cannot be dropped, even when $I$ satisfies $G_{s}$ (Examples~\ref{mystery module bis} and \ref{can't drop deformation}).

Applying Theorem~\ref{general duality} with $g+v = s-1$ we obtain a result extending Theorem~\ref{Huneke} under the additional hypothesis that the pair
$(R, I)$ admits a ``good'' deformation:

\begin{corollary}\label{almost optimal}
Assume that $(R,I)$ has a deformation
$(\tR,\tI)$ such that $\tI$ satisfies the condition $G_{s}$ and
 the Koszul homology $H_{i}(\tI)$ is Cohen-Macaulay for $0\leq i\leq t$.
If $I$ satisfies $G_{s-1}$, then
$$
I^{u}/JI^{u-1}\cong \Hom_{R} (I^{t+1-u}/JI^{t-u}, \omega_{R/K})
$$
for $1\leq u\leq t$.

\smallskip

\end{corollary}

\medskip

\section{Preliminary results}
\def\Q{{\rm Q}}

We will repeatedly use a number of results from \cite{U}. For the reader's convenience we gather them here. Unless stated otherwise,
we will assume that $R$ is a
local Gorenstein ring of dimension $d$, $I$ is an ideal of codimension $g$, and
$K = J:I$ is an $s$-residual intersection of $I$ for some $s \geq g$. As before, we write $t=s-g$, and when referring
to the Weak, Standard, or Strong Hypotheses, we mean that these hypotheses hold with respect to
$s$.
\begin{theorem}\label{basic 1} If $I$ satisfies the Strong Hypothesis,
then
$R/K$ is
Cohen-Macaulay of codimension $s$ with $\omega _{R/K} \cong I^{t+1}/JI^{t}$.
\end{theorem}
\begin{proof} This is
 \cite[2.9]{U}.
\end{proof}

\vspace{0.0cm}

\begin{proposition}\label{codimension}
If $I$ satisfies the Standard Hypothesis, then the ideal $K$ is unmixed
of codimension exactly $s$.
\end{proposition}

\begin{proof}
One uses Theorem \ref{basic 1} and  \cite[1.7(a)]{U}.
\end{proof}

\vspace{0.0cm}

\begin{proposition}\label{new basic}
 Let $J= (a_{1}, \dots, a_{s})$. For $0 \leq u\leq s$ write $J_{u} = (a_{1}, \dots, a_{u}),\ K_{u} = J_{u}:I, R_{u}= R/K_{u},$
 and assume that $K_{u}$ is a geometric $u$-residual intersection of $I$ whenever $g \leq u<s$.

\begin{enumerate}
 \item If $I$ satisfies the Weak Hypothesis and $u \geq 1$, then the element $a_{u}$ is regular on $R_{u-1}$ and
 $K_{u}R_{u-1} = a_{u}R_{u-1}:IR_{u-1}$.

\item If $I$ satisfies the Standard Hypothesis and $2 \leq j \leq t+1 \, $; or if
 $I$ satisfies the Strong Hypothesis and $2 \leq j \leq t+2 \, $;
 then there are exact sequences
 $$
 0\rTo \frac{I^{j-1}}{J_{u-1}I^{j-2}} \rTo^{a_{u}} \frac{R}{J_{u-1}I^{j-1}} \rTo\frac{R}{J_{u}I^{j-1}} \rTo 0
$$
for $u \geq 1$.

\item If $I$ satisfies the Standard Hypothesis and $1 \leq j \leq t \, $; or if
 $I$ satisfies the Strong Hypothesis and $1 \leq j \leq t+1 \, $;
 then
$$\depth I^j/J_uI^{j-1} \geq \min\{d-u, \dim R/I -j+2\} \, .$$
In particular, $I^{j}/J_{u}I^{j-1}$ is a maximal Cohen-Macaulay $R_{u}$-module if
in addition $j \leq u-g+2$.

\item If $I$ satisfies the Standard Hypothesis and $1 \leq j \leq t+1 \, $; or if
 $I$ satisfies the Strong Hypothesis and $1 \leq j \leq t+2 \, $;
 then
$$\depth R/J_uI^{j-1} \geq \min\{d-u, \dim R/I -j+2\} \, .$$

 \item If $I$ satisfies the Standard Hypothesis, $u < s$, and $1 \leq j \leq t +1\, $; or if
 $I$ satisfies the Strong Hypothesis, $K$ is a geometric $s$-residual intersection, and $1 \leq j \leq t+2 \, $;
 then
 $$I^{j}\cap K_{u} = J_{u}I^{j-1} .$$
\end{enumerate}
\end{proposition}

\begin{proof} First notice that if $u<g$ then $K_u=J_u$ is generated by the regular sequence
$a_1, \ldots, a_u$. Now
part (1) follows from Theorem \ref{basic 1} and \cite[1.7 (a),(c)]{U},
part (2) is a consequence of Theorem \ref{basic 1} and \cite[2.7(a)]{U}, and
item (3) follows from Theorem \ref{basic 1} and \cite[2.7(b)]{U}.

We now prove (4). The assertion for $j=1$ follows from Theorem \ref{basic 1} and \cite[1.7(b)]{U}.
Thus we may assume that $j \geq 2$.
We show part (4) by induction on $u$. The assertion is obvious
for $u=0$. If $1 \leq u \leq s$, we apply the exact sequence
of part (2), the depth estimate of part (3), and the induction hypothesis.

If
$I$ satisfies the Strong Hypothesis, $K$ is a geometric residual intersection and $1 \leq j \leq t+2$, 
then
part (5) follows from Theorem \ref{basic 1} and \cite[2.7(c)]{U}. 
If on the other
hand $I$ satisfies the Standard Hypothesis, $u<s$ and $1 \leq j \leq t+1$, then $I$ satisfies
the Strong Hypothesis with respect to $s-1$, $K_{s-1}$ is a geometric residual intersection and $1 \leq j \leq (s-1-g)+2$, so
the assertion follows from the previous case.
\end{proof}

\vspace{0.0cm}

\begin{proposition}\label{red lemma}
Suppose that $R$ satisfies Serre's condition $(R_{s-1})$ and contains a field of characteristic 0. Let $\seq a {s-1}$ be general elements of $J$. Set $J_{u} = (\seq a u)$, $K_{u}=J_{u}:I$, and $R_{u} = R/K_{u}$. If $R$ is reduced and $I$ satisfies
the Weak Hypothesis with respect to $s$, then the factor ring $R_{u}$ is reduced and equidimensional of
codimension $u$ for every $u<s$.
\end{proposition}

 \def\Q{{\rm Q}}
\begin{proof} Again, if $u < g$ then $K_{u}=J_{u}$ is generated by the regular sequence $a_{1}, \ldots, a_{u}$.
If $g \leq u < s$ then $K_{u}$ is a geometric $u$-residual intersection by Lemma \ref{principal ideal}
and hence this ideal is unmixed of codimension $u$ according to Proposition \ref{codimension} because
$I$ satisfies the Standard Hypothesis with respect to $u$.
In either case, $K_{u}$ is unmixed of codimension $u$ and $I$ is not contained in any of the minimal primes
of $K_u$.

Let $P$ be any of these minimal primes. To show that $R_{u}$ is reduced it suffices to prove that the ring
$(R_u)_P$ is regular.
Since $\codim(J:I)\geq s > u = \codim P$ and $P$ does not contain $I$, it follows that $P$ cannot contain $J$
either. Since the elements $a_{1}, \ldots, a_{u}$ are general in $J$, the
 local Bertini Theorems (\cite[4.6]{F})  show that  $(R/(a_{1}, \ldots, a_{u}))_{P}$  is regular.  But this ring is
$(R_u)_P$, again since $P$ does not contain $I$.
\end{proof}

\vspace{0.0cm}

\section{Connecting the canonical module with powers of $I$}
We next explain the maps that connect powers of $I$ with
the canonical module, refining Theorem~\ref{basic 1}. As we shall see, these maps are defined under a certain assumption
that is satisfied under the Standard Hypothesis, but also in some cases of geometric residual intersections.
Unless stated otherwise, the general assumptions of Section 3 are still in effect.

\def\o{{\omega}}

\begin{theorem}\label{mu definition} Let $a_{1},\dots, a_{s}$ be generators of $J$ and,
for $0 \leq u\leq s$, let $J_{u}=(a_{1},\dots, a_{u})$ and  $K_{u} = J_{u}:I$.
Assume that $\codim ((K_{u-1},a_{u}):I) = u$ whenever $1 \leq u \leq s$ and that
$\codim(I+K_{u}) \geq u+1$ whenever $0 \leq u < s$.

For every $u$ with $0 \leq u \leq s$ one has $\codim K_u =u$. Set
$R_{u} = R/K_{u}$ and $R'_u = R/K'_u$, where $K'_u$ denotes the unmixed part of $K_u$ of codimension $u$.
For every $u$ with $g \leq u \leq s$, there are maps
$$
\frac{I^{u-g+1}}{J_{u}I^{u-g}}\rTo^{\mu_{u-g}} \omega_{R_u}
$$
defined inductively$:$

\begin{enumerate}
\item{$u=g:$\ } $\mu_{0}$ is the map induced by the inclusion of $I$ into the double annihilator
$$
I/J_g \hookrightarrow  \frac{ \left(J_g: \bigl(J_g: I\bigr)\right)}{J_g}  =  \omega_{R_{g}}
$$
\item{$s\geq u>g:$\ } $\mu_{u-g}$ is the map obtained from $\mu_{u-g-1}$ and an embedding
$$
I\o_{R_{u-1}}/a_{u}\o_{R_{u-1}} \hookrightarrow \o_{R_{u}}
$$
obtained from the diagram of homomorphisms of rings
\begin{diagram}
R_{u-1}&\rOnto^{\pi_{1}}&R'_{u-1}&\rOnto^{\pi_{2}}&R'_{u-1}/(a_{u})\\
&&&&\uOnto^{\pi_{3}}\\
&&&&R_{u-1}/(a_{u})&\rOnto^{\pi_{4}} R_{u-1}/((a_{u}):IR_{u-1})\\
&&&& &\uOnto^{\pi_{5}}\\
&&&&&R_u 
\end{diagram}
as explained in the proof.
\end{enumerate}
If $I$ satisfies the Standard Hypothesis with respect to $s$, then the map $\mu_{s-g}$ is an injection, while if $I$ satisfies the Strong Hypothesis
with respect to $s$, then $\mu_{s-g}$ is an isomorphism.
\end{theorem}

 \begin{proof} We first show that $\codim K_u=u$ for $0 \leq u \leq s$ and we compute
 the codimensions of all the rings in the diagram.

Since the codimension of the ideal $K_0$ is obviously $0$, we assume that $1 \leq u \leq s$.
By Lemma~\ref{funny hypothesis},
the codimension of $K_u$ is at least $u$.
As $K_{u}\subset (K_{u-1}, a_{u}):I$ and the second ideal has codimension $u$ by hypothesis, we see that
the codimensions of the two ideals are exactly $u$. Thus, the rings $R_{u-1}$ and $R'_{u-1}$ have codimension $u-1$
in $R$, and the rings $R_u$ and $R_{u-1}/((a_{u}):IR_{u-1})$ have codimension $u$.
 
 We now claim that $a_{u}$ is not in any codimension $u-1$ prime $P$ containing $K_{u-1}$. Since $I+K_{u-1}$ has codimension $\geq u$ by hypothesis, we have $I\not \subset P$,
and since $K_u$ has codimension $u$, we see that
$I_{P} = (J_{u})_{P}$ and therefore $J_u \not\subset P$. As $J_{u-1}\subset P$, it follows that $a_{u}\notin P$. From this we see that the rings
$R_{u-1}/(a_{u})$ and $R'_{u-1}/(a_{u})$ have codimension $u$, and moreover $a_{u}$ is a nonzerodivisor on $R'_{u-1}$.

We take the map $\mu_{0}$ to be the natural inclusion.
Moreover this map is an isomorphism if the Strong Hypothesis holds
since then $R/I$ is Cohen-Macaulay. Therefore we assume from now on that $u>g$.

The map $\pi_{1}$ induces an isomorphism $(\pi_{1}^{\vee})^{-1}:\o_{R_{u-1}} \rTo^{\sim} \o_{R'_{u-1}}$.
Since $a_{u}$ is a nonzerodivisor on $R'_{u-1}$, the connecting homomorphism of $\Ext_{R}(-, \omega_{R})$ applied to the
exact sequence
$$
0\rTo R'_{u-1} \rTo^{a_{u}} R'_{u-1} \rTo^{\pi_{2}} R'_{u-1}/(a_{u}) \rTo 0
$$
yields an embedding $\sigma_{2}: \o_{R'_{u-1}}/a_{u}\o_{R'_{u-1}} \hookrightarrow \o_{R'_{u-1}/(a_{u})}.$
The map $\pi_{3}$ induces an embedding $\pi_{3}^{\vee}: \o_{R'_{u-1}/(a_{u})}\hookrightarrow \o_{R_{u-1}/(a_{u})}$.

For simplicity of notation we set
$$\o := \o_{R_{u-1}/(a_{u})}\quad \hbox{and}\quad
H:=(a_{u}):IR_{u-1}\subset R_{u-1}.
$$

Multiplying by $I$, we see that the maps $(\pi_{1}^{\vee})^{-1},\ \sigma_{2}$ and $\pi_{3}^{\vee}$ together induce an embedding
$$
 I\o_{R_{u-1}}/a_{u}\o_{R_{u-1}} \hookrightarrow I\o.
$$

On the other hand,
$$
I\o \subset 0:_{\o}H = \o_{R_{u-1}/H},
$$
and combining these two embeddings we obtain
\begin{equation}\label{eqn 2}
 I\o_{R_{u-1}}/a_{u}\o_{R_{u-1}} \hookrightarrow \o_{R_{u-1}/H}.
\end{equation}

Finally,
the map $\pi_{5}$ induces an embedding $\pi_{5}^{\vee}: \o_{R_{u-1}/H} \hookrightarrow \o_{R_{u}}$,
which together with the map in (1) gives an embedding
\begin{equation}
I\o_{R_{u-1}}/a_{u}\o_{R_{u-1}} \hookrightarrow \o_{R_{u}}.
\end{equation}

By induction, we may assume that the process just explained induces a
map $I^{u-g}/J_{u-1}I^{u-g-1} \to \o_{R_{u-1}}$, and thus we obtain a map
$$
\frac{I^{u-g+1}+J_{u-1}I^{u-g-1}}{a_{u}I^{u-g}+J_{u-1}I^{u-g-1}} \rTo \o_{R_{u}}.
$$
The left hand side is obviously a homomorphic image of $I^{u-g+1}/J_{u}I^{u-g}$, and this gives the desired homomorphism
$$
\mu_{u-g}: I^{u-g+1}/J_{u}I^{u-g} \rTo \o_{R_{u}}.
$$
\bigbreak

We now show by induction on $u>g$ that if $I$ satisfies the Standard Hypothesis or the Strong
Hypothesis, then $\mu_{u-g}$ is an injection or an isomorphism, respectively.
In either case $I$ satisfies the Strong Hypothesis
with respect to $u-1$, so in particular $\mu_{u-g-1}: I^{u-g}/J_{u-1}I^{u-g-1}\to \omega_{R_{u-1}}$ is an isomorphism
by the induction hypothesis. Multiplying
by $I$ and factoring out $a_{u}I^{u-g}$ we get an induced isomorphism
$$
\frac{I^{u-g+1}}{a_{u}I^{u-g}+ (J_{u-1}I^{u-g-1}\cap I^{u-g+1})} \rTo^{\sim} \frac{I\o_{R_{u-1}}}{a_{u}\o_{R_{u-1}}} \, .
$$

By (2), the right hand side embeds in $\omega_{R_{u}}$. So to prove the injectivity of $\mu_{u-g}$ it suffices to show that
$$
a_{u}I^{u-g}+ (J_{u-1}I^{u-g-1}\cap I^{u-g+1}) = J_{u}I^{u-g}.
$$
The right hand side is obviously contained in the left hand side, so it remains to prove the opposite inclusion.
We trivially have
$$
J_{u-1}I^{u-g-1}\cap I^{u-g+1} \subset I^{u-g+1} \cap K_{u-1},
$$
and Proposition \ref{new basic}(5) gives $I^{u-g+1} \cap K_{u-1}= J_{u-1}I^{u-g}$. 
This concludes the proof that $\mu_{u-g}$ is an injection.

We now show that if $I$ satisfies the Strong Hypothesis, then $\mu_{u-g}$ is a
surjection. To this end it suffices to prove that the map in (2) is a surjection. 
Since $\pi_{5}$ is an isomorphism according to 
Proposition \ref{new basic}(1),
it remains to show that the map in (1) is surjective.
We summarize the argument in the proof of \cite[2.9(b)]{U}.
Recall that $R_{u-1}$ is Cohen-Macaulay by Theorem~\ref{basic 1}.

We first prove that $I\omega_{R_{u-1}}$ is $\omega_{R_{u-1}}$-reflexive.
By induction, $\o_{R_{u-1}} \cong I^{u-g}/J_{u-1}I^{u-g-1}$. Proposition \ref{new basic}(5)
shows that $J_{u-1}I^{u-g-1}=I^{u-g}\cap K_{u-1}$ and therefore
$I^{u-g}/J_{u-1}I^{u-g-1} \cong I^{u-g}R_{u-1}$. It follows that $I\o_{R_{u-1}} \cong I^{u-g+1}R_{u-1}$.
But again by Proposition \ref{new basic}(5),  $I^{u-g+1}R_{u-1} \cong I^{u-g+1}/J_{u-1}I^{u-g} $. Putting this together, we obtain
$$
I\o_{R_{u-1}}  \cong I^{u-g+1}/J_{u-1}I^{u-g}.
$$
By Proposition \ref{new basic}(3), $I^{u-g+1}/J_{u-1}I^{u-g}$ is a
maximal Cohen-Macaulay $R_{u-1}$-module and thus $I\omega_{R_{u-1}}$ is $\omega_{R_{u-1}}$-reflexive, which we write as
$I\omega_{R_{u-1}} = (I\omega_{R_{u-1}})^{\vee \vee}$.

We deduce that
$$I\omega_{R_{u-1}}/a_u\o_{R_{u-1}} = (I\omega_{R_{u-1}})^{\vee \vee}/a_u\o_{R_{u-1}}=\o_{R_{u-1}/H}\, ,$$
where the last identification holds according to \cite[2.1(a)]{U}. Therefore 
the map in (1) is surjective.
This concludes the proof.
\end{proof}

 \smallskip

 \begin{lemma}\label{funny hypothesis} Let $R$ be a Noetherian ring, let
$J\subset I$ be ideals, and let $a\in R$ be an element. If
$$
\codim\ (J:I,a):I\geq u  \hbox{ \quad and \quad} \codim  \big( I +((J,a):I)\big) \geq u \, ,
$$
then
$$
\codim\ (J,a):I \geq u.
$$
\end{lemma}
\begin{proof} One sees that
$$
\big ((J:I,a):I \big) \big(    I +((J,a):I) \big) \subset (J:I,a) + ((J,a):I)  \subset (J,a):I  .
$$
\end{proof}
\smallskip

For future use we record the following statements, proved in the course of the proof of Theorem~\ref{mu definition}.

\begin{corollary}\label{reduction lemma} With the notation and assumptions of Theorem~$\ref{mu definition}$, assume that
$I$ satisfies the Strong Hypothesis.
For $0<g\leq u\leq s$ the rings $R_{u-1}/(a_u)$ and $ R_u$ are Cohen-Macaulay of  dimension $d-u$, and
the surjection $R_{u-1}/(a_u) \twoheadrightarrow R_u$ induces an inclusion of
canonical modules
$$
I^{u-g+1}/J_uI^{u-g}\cong  \omega_{R_u} \hookrightarrow \omega_{R_{u-1}/(a_u)}\cong
I^{u-g}R_{u-1}/(a_uI^{u-g}R_{u-1})
$$
that is compatible with the natural inclusion $I^{u-g+1}\subset I^{u-g}.$
\end{corollary}

\smallskip

\begin{remark}[The Graded Case]\label{graded case} Suppose that $R$ is a standard graded polynomial ring
$k[\seq x d]$. Suppose further that the ideal $I$ is homogeneous
and that the generators $\seq a s$ of $J$ are homogeneous of degrees $\seq \delta s$.
In this setting the construction of Theorem~\ref{mu definition} yields a homogeneous map
$$
I^{u-g+1}/J_{u}I^{u-g}\rTo^{\mu_{u-g}} \omega_{R_{u}}(d-\sum_{j=1}^{u}\delta_{j}).
$$
\end{remark}

\medskip

\section{Proofs of the duality theorems from Section ~\ref{Sduality}}\label{Sproofs}
We follow a suggestion of the referee, and include the statement of each theorem from Section 2 before its proof. Theorems from Section 2 retain the numbering that they had there.
Unless specified otherwise, $I$ will again denote an ideal of codimension $g$ in a local Gorenstein ring $R$, the ideal $K=J:I$ is assumed to be an $s$-residual intersection, and we set $t=s-g$.

\medbreak\noindent {\bf Theorem~\ref{Huneke}.}{\it\ \
Suppose that $R/I$ is Cohen-Macaulay of codimension $g$ and $J = (a_1,\dots, a_{g+1})\subset I$ is such that
$K = J:I$ has codimension $g+1$, then
the $R/K$-module $I/J$ is self-dual; that is,
$$
I/J\cong \Hom_{R}(I/J, \omega_{R/K}).
$$}
\begin{proof}(Huneke)
We may suppose that $J = (a_1,\dots, a_g, b)$
where $a_1,\dots, a_g$ form a regular sequence.
Factoring out $a_1,\dots, a_g$ we may assume $g=0$.

Let $L = 0:b$, and consider the short exact sequence
$$
0\rTo R/L \rTo^{b} R \rTo R/(b) \rTo 0 \, .
$$
Dualizing into $R$ we obtain an exact sequence
$$
R \rTo^\beta 0:L \rTo \Ext^1_R(R/(b), R)\rTo 0 \, .
$$
The image of $\beta$ is the ideal generated by $b$.
Also, we claim that $0:L = I$. Because $R$ is Gorenstein
the ideal $0:L = 0:(0:b)$
is the unmixed part of $(b)$, which is equal to $I$ because
$I$ is unmixed of codimension zero and $(b):I$ has positive codimension in $R$. Putting these two
observations together, we get
$$
I/(b) \cong \Ext^1_R(R/(b), R) \, .
 $$

On the other hand, because $R/I$ is a maximal Cohen-Macaulay $R$-module, and $R$ is Gorenstein,
we have
$$
\Ext^1_R(R/I, R) = \Ext^2_R(R/I, R) = 0 \, ,
$$
 so from the short exact sequence
$$
0\rTo I/(b) \rTo R/(b) \rTo R/I \rTo 0
$$
we get
$$
\Ext^1_R(R/(b), R) \cong \Ext^1_R(I/(b), R) \, .
$$

Since $K$ is an ideal of codimension 1 in the Gorenstein ring $R$ and $K$ annihilates $I/(b)$, 
we have $\Ext^1_R(I/(b), R) \cong \Hom_R(I/(b), \omega_{R/K})$,
and since we already showed that $I/(b) \cong \Ext^1_R(R/(b), R)$, we conclude that
$I/(b) \cong \Hom_R(I/(b), \omega_{R/K})$ as required.
\end{proof}

\smallskip

For the proof of Theorem~\ref{duality} we will need:

\begin{lemma}\label{photo}
In addition to the Standard Hypothesis assume that the residue field $k$ is infinite. Write
$d$ for the dimension of $R$ and let
$x_{1},\dots, x_{d-s}$ be general elements in the maximal ideal. For any $1\leq u\leq t$ one has$:$
\begin{enumerate}
\item The elements $x_{1},\dots,x_{d-s}$ form a regular sequence on $R$ and on $R/I^{u}$.
\item The image $\Ibar$ of $I$ in $R/(x_{1},\dots,x_{d-s})$ satisfies the condition $G_{s}$.
\item We adopt the notation of Theorem~$\ref{duality}$. 
The image $\overline J$  defines an $s$-residual intersection
$\overline J: \Ibar$ in $R/(x_1,\dots, x_{d-s})$.  
Moreover, if $m(\Ibar, u,t)$ is a perfect pairing, then so
are $m(I,u,t)$ and $\mu_{t}\circ m(I,u,t)$.
\end{enumerate}
\end{lemma}

\begin{proof}
\noindent (1): By the Standard Hypothesis
$$\depth(R/I^{u})\geq \dim(R/I)-u+1\geq \dim(R/I)-t+1 = d-s+1.$$
In particular,
the elements $x_{1},\dots,x_{d-s}$ form a regular sequence on $R$ and on $R/I^{u}$.

\smallbreak\noindent (2): The condition $G_{s}$ is equivalent to the condition that the codimension of 
$I + \,$Fitt$_{i-1}(I)$ is at least $i$ for $1 \leq i \leq s$. The Fitting ideals of the image $\Ibar$ of $I$ in $R/(x_{1},\dots,x_{d-s})$
contain the image of the Fitting ideals, and so the codimensions of  $\Ibar + \,$Fitt$_{i-1}(\Ibar)$ satisfy  the same inequalities because
the elements $\seq x {d-s}$ are general and $\dim R/(\seq x {d-s}) = s$.

\smallbreak\noindent (3): By Proposition~\ref{codimension}, the codimension of $K$ is exactly $s$. Let $y_{1},\dots, y_{s}$ be a
regular sequence inside $K$, and set $A = R/(y_{1},\dots, y_{s})$. Note that $x_{1}, \dots, x_{d-s}$ is a regular sequence on $A$.

We recall the map $\mu_t$ of Theorem \ref{mu definition}, which in an embedding under the present assumptions. The maps
\begin{align*}
I^{u}/JI^{u-1} \otimes_{R}I^{t+1-u}/JI^{t-u} \rTo^{m(I,u,t)} I^{t+1} /JI^{t} \rInto^{\mu_{t}} \omega_{R/K}\rInto \omega_{A}
\end{align*}
induce maps
\begin{align*}
I^{t+1-u}/JI^{t-u}&\rTo^{\alpha} \Hom_{R}(I^{u}/JI^{u-1}, I^{t+1}/JI^{t}) \\
&\rInto^{\beta}
\Hom_{R}(I^{u}/JI^{u-1},\omega_{R/K}) \\
&\rTo^{\cong}
\Hom_{R}(I^{u}/JI^{u-1},\omega_{A}),
\end{align*}
where the last map is an isomorphism by Hom-tensor adjointness. We must show that under our hypothesis $\alpha$ and $\beta$ are
both isomorphisms.

By Proposition~\ref{new basic} (3) the module $I^{u}/JI^{u-1}$ is a maximal Cohen-Macaulay $A$-module. Since $A$
is Cohen-Macaulay,
$\Hom_{R}(I^{u}/JI^{u-1},\omega_{A})$ is a maximal Cohen-Macaulay $A$-module too. Thus $\seq x {d-s}$
form a regular sequence on this module.

Let $\Abar = A/(x_{1}, \dots, x_{d-s})$ and $\Rbar = R/(x_{1}, \dots, x_{d-s})$; we write $\Ibar, \Jbar, \Kbar$ for the images
of $I,J,K$ in $\Rbar$, respectively. We will show that
$(\beta\alpha)\otimes_{R}\Rbar$ is an isomorphism. This implies that $\beta\alpha$ is surjective, and thus $\beta$ is surjective and
consequently $\beta$ is an isomorphism. It follows that $\alpha$ is also surjective, and
$\seq x {d-s}$ is a regular sequence on the image of $\alpha$. Because $\alpha\otimes_{R}\Rbar$ is
an injection and $x_{1},\dots,x_{d-s}$ is a regular sequence on the image of $\alpha$, it follows that $\alpha$ is a monomorphism.

It remains to show that $(\beta\alpha)\otimes_{R}\Rbar$ is an isomorphism. The ideal $\Ibar$ satisfies the Standard Hypothesis
by items (1) and (2).  Since $s =\dim \Rbar$, the Standard Hypothesis is the same as the Strong Hypothesis in this case.
The ideal $\Kbar$ has codimension $s$ and is contained in $\Jbar:\Ibar$, hence
$\Jbar:\Ibar$ is an $s$-residual intersection.
 Arguing as above, there are maps
\begin{align*}
\Ibar^{t+1-u}/\Jbar\Ibar^{t-u}&\rTo^{\abar} \Hom_{\Rbar}(\Ibar^{u}/\Jbar\Ibar^{u-1}, \Ibar^{t+1}/\Jbar\Ibar^{t}) \\
&\rInto^{\bbar}
\Hom_{\Rbar}(\Ibar^{u}/\Jbar\Ibar^{u-1},\omega_{\Rbar/(\Jbar \, : \, \Ibar)}) \\
&\rTo^{\cong}
\Hom_{\Rbar}(\Ibar^{u}/\Jbar\Ibar^{u-1},\omega_{\Abar}),
\end{align*}
induced by $m(\Ibar,u,t)$ and $\mu_{t}$. By assumption, $\abar$ is an isomorphism. Moreover, since
$\Ibar$ satisfies the Strong Hypothesis, $\bbar$ is an
isomorphism by Theorem~\ref{mu definition}.

Because $x_{1},\dots, x_{d-s}$ form a regular sequence on $R/I^{u}$ by item (1), it follows that
$\Ibar^{u}/\Jbar\Ibar^{u-1}\cong I^{u}/JI^{u-1}\otimes_{R}\Rbar$. Further

\begin{align*}
 \Hom_{\Rbar}(\Ibar^{u}/\Jbar\Ibar^{u-1},\omega_{\Abar}) &\cong
\Hom_{\Rbar}(I^{u}/JI^{u-1}\otimes_{R}\Rbar,\omega_{A}\otimes_{R}\Rbar)\\ &\cong
\Hom_{R}(I^{u}/JI^{u-1},\omega_{A})\otimes_{R}\Rbar,
\end{align*}
where the second isomorphism holds because $I^{u}/JI^{u-1}$ is a maximal Cohen-Macaulay $A$-module.

In the following commutative diagram
$$
\begin{diagram}[small]
I^{t+1-u}/JI^{t-u}\otimes_{R}\Rbar & \rTo^{(\beta\alpha)\otimes_{R}\Rbar} & \Hom_{R}(I^{u}/JI^{u-1},\omega_{A})\otimes_{R}\Rbar\\
\dIso^{\cong}&&\dIso_{\cong}\\
\Ibar^{t+1-u}/\Jbar\Ibar^{t-u}& \rTo^{\cong} &\Hom_{\Rbar}(\Ibar^{u}/\Jbar\Ibar^{u-1},\omega_{\Abar})
\end{diagram}
$$
we can take the vertical maps and the bottom horizontal map to be the isomorphisms established above. Thus
$(\beta\alpha)\otimes_{R}\Rbar$ is an isomorphism as required.
\end{proof}

\medbreak\noindent 
{\bf Theorem~\ref{duality}.}{\it\ \
Under the Standard Hypothesis, Theorem~\ref{mu definition} applies to give an injective  map
$\mu_{t}: I^{t+1}/JI^{t}\to \omega_{R/K}$. For $1\leq u\leq t$, both the multiplication map
$$
m(I,u,t): I^{u}/JI^{u-1}\otimes I^{t+1-u}/JI^{t-u} \rTo^{\rm mult}  I^{t+1}/JI^{t}
$$
and the composition $\mu_{t}\circ m(I,u,t)$
are perfect pairings.
}
\begin{proof}
The injectivity of $\mu_{t}$ was proven in Theorem~\ref{mu definition}, so it suffices
to prove the duality statements. We proceed by induction on $t \geq 0$, the case $t=0$ being vacuous.

We may assume that the residue field $k$ is infinite. We may harmlessly replace $R$ by $R[[x]]$ and replace $I,J$ by $(I,x), (J,x)$. In this new setting we have $g>0$. After proving the result in this new setting, the original result is recovered by taking the degree 0 part with respect to $x$. By Lemma~\ref{photo} we may further assume that $d=s$.

In this case the extra strength of the Strong Hypothesis is vacuous. Thus we may apply Theorem~\ref{basic 1} to deduce
that $I^{t+1}/JI^{t}\cong \omega_{R/K}$. Further, $I/J$ has finite length, and it follows that the lengths of the modules
$I^{u}/JI^{u-1}$ and $\Hom(I^{u}/JI^{u-1}, I^{t+1}/JI^{t})$ are equal. We will prove that 
\begin{equation}
JI^{t}:I^{u} \subset JI^{t-u}
\end{equation}
for all $1\leq u\leq t$. It will follow that the map
$$
I^{t+1-u}/JI^{t-u} \rTo \Hom(I^{u}/JI^{u-1}, I^{t+1}/JI^{t})
$$
induced by multiplication is injective, and thus
\begin{align*}
\length(I^{t+1-u}/JI^{t-u}) &\leq \length \Hom(I^{u}/JI^{u-1}, I^{t+1}/JI^{t})\\
 &= \length (I^{u}/JI^{u-1}).
\end{align*}
Since this set of  inequalities is symmetric under interchanging $u$ and $t+1-u$, it follows that
$\length(I^{u}/JI^{u-1}) = \length (I^{t+1-u}/JI^{t-u})$, and thus the injective map above is an isomorphism.

It remains to prove equation (3).
We use Lemma~\ref{principal ideal} with ${\frak a}=J$, and we adopt the notations $J_{s-1}$, $a_{s}$ and $K_{s-1}$
from that Lemma. We write $\Rbar=R/K_{s-1}$ and $\Ibar = I\Rbar$. By  Proposition~\ref{new basic}(5),
$\Ibar^{u}\cong I^{u}/J_{s-1}I^{u-1}$. By the induction hypothesis $m(I, u,t-1)$ is a perfect pairing.
That is, for $1\leq u\leq t-1$ the natural maps
$$
\Ibar^{t-u} \rTo \Hom_{\Rbar}(\Ibar^{u}, \Ibar^{t})
$$
are isomorphisms, and this condition  holds also for $u =t$ because $\Rbar$ is
Cohen-Macaulay with canonical module $\Ibar^{t}$, by Theorem~\ref{basic 1}
and Proposition~\ref{new basic}(5). 
 Recall that $J\Rbar = a_{s}\Rbar \subset \Ibar$. By Proposition~\ref{new basic}(1), $a_{s}$ is regular on $\Rbar$.
Since the ideal $\Ibar$ contains a nonzerodivisor, 
there is a natural isomorphism $\Hom_{\Rbar}(\Ibar^{u}, \Ibar^{t}) \cong \Ibar^{t}:_{Q(\Rbar)}\Ibar^{u}$, where $Q(\Rbar)$ denotes
the total ring of quotients of $\Rbar$. Therefore
\begin{equation*}\label{Gorenstein}
\Ibar^{t}:_{Q(\Rbar)}\Ibar^{u} = \Ibar^{t-u}.
\end{equation*}
Since $J\Rbar$ is generated by the nonzerdivisor $a_{s}$,
$$
(J\Ibar^{t}):_{Q(\Rbar)}\Ibar^{u} = (a_{s}\Ibar^{t}):_{Q(\Rbar)}\Ibar^{u} = a_{s}(\Ibar^{t}:_{Q(\Rbar)}\Ibar^{u})= a_{s}\Ibar^{t-u}.
$$
In particular,
$$
(J\Ibar^{t}):_{\Rbar}\Ibar^{u} \subset J\Ibar^{t-u},
$$
and hence
\begin{equation}\label{over R}
(JI^{t}):_{R}I^{u} \subset JI^{t-u}+K_{s-1}.
\end{equation}

On the other hand, our assumptions on $I$ imply that locally on the punctured spectrum of $R$, the associated graded ring $\gr_{I}(R)$ is Cohen-Macaulay
(Theorem \ref{basic 1}, \cite[3.4]{HVV}, and \cite[6.1]{HSV2}).
Since $g>0$ it follows that locally on the punctured spectrum of $R$, the irrelevant ideal of $\gr_{I}(R)$ has positive grade and therefore
$I^{t+1}:I^{u} = I^{t+1-u}$. Since by the Standard Hypothesis the maximal ideal is not an associated prime of $R/I^{t+1-u}$,
we conclude that $I^{t+1}:I^{u} = I^{t+1-u}$ globally in $R$.
In particular, $JI^{t}:I^{u}\subset I^{t+1-u}$, so equation (\ref{over R}) gives
$$
JI^{t}:I^{u}\subset JI^{t-u}+ K_{s-1}\cap I^{t+1-u}.
$$
Finally, by Proposition~\ref{new basic}(5),
$$
K_{s-1}\cap I^{t+1-u} = J_{s-1}I^{t-u},
$$
which completes the proof of (3).
\end{proof}

\vspace{0.0cm}

\medbreak\noindent {\bf Corollary~\ref{Rees is Gor}.}{\it\ \
Under the Strong Hypothesis, the ring
\begin{align*}
\RR &:= R/K \oplus I/J \oplus I^{2}/JI \oplus \ \ldots \ \oplus I^{t+1}/JI^{t} \\
&= R[Iz]/(K, Jz, (Iz)^{t+2})
\end{align*}
is Gorenstein.
}

\begin{proof}
As a graded $R$-algebra $\RR$ is generated in degree 1 and concentrated in degrees $0,\dots, t+1$,
so  the Gorenstein property is equivalent to the statements:
\begin{enumerate}
\item $\RR$ is a maximal Cohen-Macaulay module over $\RR_{0}$.
\item $\RR_{t+1} = \omega_{\RR_{0}}.$
\item The pairings $\RR_{u}\otimes \RR_{t+1-u}\to \RR_{t+1}$ induced by multiplication are perfect for $u = 1,\dots, t$.
\end{enumerate}
Here items (2) and (3) are equivalent to the existence of an isomorphism of graded $\RR$-modules
$$
\Hom_{\RR_{0}}(\RR, \omega_{\RR_{0}})(-t-1)\cong \RR.
$$
Item (1)  follows from Theorem~\ref{basic 1} and  Proposition~\ref{new basic}(3). Item (2) follows
from Theorem~\ref{basic 1}, while item (3) is the conclusion of Theorem~\ref{duality}.
\end{proof}

For the proof of Theorem~\ref{associated graded} we will use the following general result:

\begin{proposition}\label{gor from gor}
Let $R$ be a local Cohen-Macaulay ring, let $I\subset R$ be an ideal of positive codimension, 
and let $t \geq 0$ be an integer. If the truncated
Rees ring $\RR(I)/\RR(I)_{\geq t+2}$ is Gorenstein, then so is the truncated associated graded
ring $\gr_{I}(R)/\gr_{I}(R)_{\geq t+1}$ and the ring $R/I^{t+1}$.
\end{proposition}
\begin{proof} We may assume that $I \neq R$.
Write $d = \dim R$ and set
\begin{align*}
A &:= \RR(I)/\RR(I)_{\geq t+2} = R\oplus I\oplus \ \ldots \ \oplus I^{t+1}, \\
B &:= \gr_{I}(R)/\gr_{I}(R)_{\geq t+1} = R/I\oplus I/I^{2}\oplus \ \ldots \ \oplus I^{t}/I^{t+1}.
\end{align*}
Since $A$ is a Cohen-Macaulay ring, finite over $R$, the ideal $I^{j}$ is a maximal Cohen-Macaulay module for $j\leq t+1$, and
 it follows that $R/I^{j}$ is a Cohen-Macaulay ring of dimension $d-1$ for
$j\leq t+1$. From this we see that $I^{j}/I^{j+1}$ is a maximal Cohen-Macaulay $R/I$-module for
$j\leq t$. Thus $B$ is a Cohen-Macaulay ring of dimension $d-1$.

To prove that $B$ is Gorenstein we will show that $\omega_{B} = \Ext^{1}_{R}(B, \omega_{R})$ is
cyclic as a $B$-module by showing that there is a surjection of $A$-modules from
the cyclic $A$-module $\omega_{A}(-1)$ to $\omega_{B}$.
The exact sequence of $A$-modules
$$
0\rTo A_{\geq 1}\rTo A\rTo R\rTo 0
$$
is split as a sequence of $R$-modules, so there is a surjection of $A$-modules
$$
\omega_{A} = \Hom_{R}(A,\omega_{R}) \twoheadrightarrow \Hom_{R}(A_{\geq 1}, \omega_{R}).
$$
On the other hand, from the exact sequence of $A$-modules
$$
0\rTo A_{\geq 1}(1)\rTo A/A_{t+1}\rTo B\rTo 0
$$
we get a map
$$
\Hom_{R}(A_{\geq 1}, \omega_{R})(-1)\rTo \Ext^{1}_{R}(B,\omega_{R}) = \omega_{B}
$$
that is surjective because $A/A_{t+1}$ is a maximal Cohen-Macaulay $R$-module.

Finally, since $B$ is Gorenstein and $B$ is the associated graded ring of $R/I^{t+1}$ with respect
to the ideal $I/I^{t+1}$, it follows that $R/I^{t+1}$ is Gorenstein as well.
\end{proof}

\medbreak\noindent {\bf Theorem~\ref{associated graded}.}{\it\ \
In addition to the Strong Hypothesis, suppose that $K=J:I$ is a geometric $s$-residual intersection.
\begin{enumerate}
\item Let $\Ibar\subset \Rbar:=R/K$ be the image of $I$.
The truncated Rees algebra
$$
\Rbar \oplus \Ibar \oplus \Ibar^{2}\oplus \ \ldots \ \oplus \Ibar^{t+1}
$$
is Gorenstein. In particular, $\Ibar^{t+1}\cong \omega_{\Rbar}$ and the multiplication maps
$\Ibar^{u}\otimes \Ibar^{t+1-u}\to \Ibar^{t+1}$ are perfect pairings for $0\leq u\leq t+1$.

\item Let $I'\subset R' := R/(K+I^{t+1})$ be the image of $I$. The associated graded ring
$\gr_{I'}(R')$ is Gorenstein.
\end{enumerate}
}

\begin{proof}
Recall that $\Rbar$ is Cohen-Macaulay according to Theorem~\ref{basic 1}. By assumption the residual intersection is geometric, so
$\Ibar$ has positive codimension in $\Rbar$ by
Proposition~\ref{codimension}. The truncated Rees algebra $\RR(\Ibar)/\RR(\Ibar)_{\geq t+2}$ is equal to $R[Iz]/(K,Jz, (Iz)^{t+2})$ by
Proposition~\ref{new basic}(5). From Corollary~\ref{Rees is Gor} we see that
this ring is Gorenstein. Thus by Proposition~\ref{gor from gor}, the truncated associated graded ring
$
\gr_{\Ibar}(\Rbar)/\gr_{\Ibar}(\Rbar)_{\geq t+1}
$
is Gorenstein. Since $R' = \Rbar/\Ibar^{t+1}$, the associated graded ring
$\gr_{I'}(R')$ is equal to $\gr_{\Ibar}(\Rbar)/\gr_{\Ibar}(\Rbar)_{\geq t+1}$,
completing the argument.
\end{proof}

\smallskip

\medbreak\noindent {\bf Theorem~\ref{general duality}.}{\it\ \
Suppose that $(R,I)$ has a deformation
$(\tR,\tI)$ such that $\tI$ satisfies the condition $G_{s}$ and
the Koszul homology $H_{i}(\tI)$ is Cohen-Macaulay for $0\leq i\leq t = s-g$.
Assume further that  $I$ satisfies the condition $G_{g+v}$ for some $(t-1)/2\leq v\leq t$.

Let $\tJ$ be a lifting of
$J$ to an ideal with $s$ generators contained in $\tI$. 
The ideal
$\tK = \tJ:\tI$ is an $s$-residual intersection of $\tI$. 
Our hypothesis implies that Theorem~$\ref{mu definition}$ holds with $\tK$ in place of $K$
and gives an isomorphism $\mu_{t}$.
The inverse $\phi: \omega_{\tR/\tK}\to \tI^{t+1}/\tJ\tI^{t}$  of  $\mu_{t}$
 induces a map $\phi': \omega_{R/K}\to I^{t+1}/JI^{t}$. We have$:$

\begin{enumerate}
\item $\phi'$ is a surjection, and is an isomorphism if $K$ is a geometric $s$-residual intersection.
\item There are perfect pairings
$$
m: I^{u}/JI^{u-1}\otimes I^{t+1-u}/JI^{t-u} \rTo \omega_{R/K}
$$
for
$$
t-v \leq u \leq v+1
$$
or, equivalently, for
$$
\frac{t+1}{2} -\varepsilon \leq u \leq \frac{t+1}{2} +\varepsilon \, ,
$$
where $\varepsilon = v-(t-1)/2$.

\item If the perfect pairing $m$ is chosen as in Figure $1$ in the proof below, then
$\phi'\circ m$ is the map induced by multiplication $I^{u}\otimes I^{t+1-u}\to I^{t+1}.$
\end{enumerate}
}

\begin{proof}
We first show that
 $\tK = \tJ:\tI$ is an $s$-residual intersection of $\tI$, that is, $\codim \tK\geq s$. To this end, note that
 $\tK R\subset K$ and, by \cite[4.1]{HU}, $K\subset \sqrt{\tK R}$. Thus
 $\codim \tK R = \codim K\geq s$. Since $\codim \tK\geq \codim \tK R$ we see that
 $\codim \tK\geq s$ as required.

The ideal $\tI$ satisfies the Strong Hypothesis, by the discussion in Section 1. If we had assumed that the residue field was
infinite,   Lemma~\ref{general position} would give the appropriate lower bounds for codimensions of ideals in
the assumptions of Theorem~\ref{mu definition}.
The lower bounds follow even without an infinite residue field from the references in the proof of Lemma~\ref{general position}.
On the other hand, the necessary upper bounds follow from 
Proposition~\ref{new basic}(1) and Proposition~\ref{codimension}.
Hence Theorem~\ref{mu definition} gives an isomorphism $\mu_{t}: \tI^{t+1}/\tJ\tI^t \to \omega_{\tR/\tK}$.

 Since $\tI$ satisfies the Strong Hypothesis, we also know from Theorem~\ref{basic 1} and Proposition~\ref{codimension}
 that $\tR/\tK$ is Cohen-Macaulay and $\tK$ has codimension exactly $s$.
 It follows that $\codim \tK = \codim \tK R= \codim K$. Thus $(\tR, \tK)$
 is a deformation of $(R,\tK R)$, and $\omega_{R/\tK R}\cong \omega_{\tR/\tK}\otimes_{\tR}R$.
 From the surjection $R/\tK R\to R/K$ and the equality of dimensions, we get an inclusion
$$
    \omega_{R/K}\hookrightarrow \omega_{R/\tK R}\cong \omega_{\tR/\tK}\otimes_{\tR}R
$$
that identifies $\omega_{R/K}$ with the set of elements of $\omega_{R/\tK R}$ that are annihilated by $K$.
From this inclusion and the isomorphism $\phi: \omega_{\tR/\tK}\to \tI^{t+1}/\tJ\tI^{t}$  of Theorem~\ref{mu definition}
we derive a map  $\phi': \omega_{R/K}\to I^{t+1}/JI^{t}$.

Next we will show that $\tI^{u} \otimes_{\tR} R \cong I^{u}$ for all $u\leq v+1$. Because $(\tR,\tI)$ is a deformation of $(R,I)$, we
may write $R = \tR/(\underline x)$, where $\underline x$ is a regular sequence on $\tR$ and on $\tR/\tI$. It suffices to show
$\underline x$ is a regular sequence modulo $\tI^{u}$ for $u$ in the given range. Since we know this for $u=1$,
we may do induction on $u$, and it is enough to show
that $\underline x$ is a regular sequence on $\tI^{u-1}/\tI^{u}$.

Fix a set of generators of $\tI$, and their images in $I$. Using these generators we define surjective maps from free modules
$\tF \to \tI$ and $F := \tF \otimes_{\tR} R\to I$ and compute Koszul homology
modules $\tH_{i}:= H_{i}(\tI)$ and $H_{i}:=H_{i}(I)$.

We now form the approximation complexes (\cite[p.470]{HSV1})
\begin{align*}
0\to \tH_{u-1}\otimes \Sym_{0}\tF\to\ \cdots\ \to \tH_{0}\otimes \Sym_{u-1}\tF \to
& \, \tI^{u-1}/\tI^{u}\to 0\\
0\to H_{u-1}\otimes \Sym_{0}F\to\ \cdots\ \to H_{0}\otimes \Sym_{u-1}F \to & \, I^{u-1}/I^{u}\to 0 \, .
\end{align*}
Since $u-1\leq v\leq t$ our hypothesis shows that the modules $\tH_{i}$ are either 0 or are
maximal Cohen-Macaulay $\tR/\tI$-modules whenever $0\leq i\leq u-1$. This implies, in the given range,
that $\underline x$ is a regular sequence on the nonzero $\tH_{i}$, that $H_{i}\cong \tH_{i}\otimes_{\tR}R$, and that the latter modules are  Cohen-Macaulay $R/I$-modules.

Since $u-1\leq v$, both $\tI$ and $I$ satisfy $G_{g+u-1}$, and it follows from \cite[the proofs of 2.5 and 2.3]{HSV1} that
both approximation complexes are exact. Since $\underline x$ is a regular sequence on all the nontrivial $\tH_{i}$ that appear,
and $H_{i}\cong \tH_{i}\otimes_{\tR}R$, the exactness of the complexes implies that $\underline x$ is a regular sequence on $\tI^{u-1}/\tI^{u}$.

This completes the argument that $\tI^{u} \otimes_{\tR} R \cong I^{u}$ for all $u\leq v+1$. From this isomorphism, we see that
$$
\tI^{u}/\tJ \tI^{u-1} \otimes_{\tR} R \cong I^{u}/JI^{u-1}.
$$

 Now let
 $$
\frac{t+1}{2} -\varepsilon \leq u \leq \frac{t+1}{2} +\varepsilon .
$$
Note that $u\leq v+1$ and $t+1-u\leq v+1$ so, by what we have just proven,
\begin{align*}
\tI^{u}/\tJ \tI^{u-1} \otimes_{\tR} R &\cong I^{u}/JI^{u-1}\\
\tI^{t+1-u}/\tJ \tI^{t-u} \otimes_{\tR} R &\cong I^{t+1-u}/JI^{t-u}.
\end{align*}
Theorem~\ref{basic 1} shows that $\omega_{\tR/\tK} \cong \tI^{t+1}/\tJ\tI^{t}$. By the argument at the beginning of this proof,
$\omega_{R/K}$ can be identified with the submodule of $\omega_{\tR/\tK}\otimes_{\tR}R$ consisting of all elements annihilated by $K$.
Thus we obtain the commutative diagram of solid arrows below:
\begin{center}
\begin{tikzpicture}
[every node/.style={scale=1.1},  auto]
\node(A) {};
\node(10) [node distance = 1cm right of A]
{$\frac{\tI^{u}}{\tJ\tI^{u-1}}\otimes \frac{\tI^{t+1-u}}{\tJ\tI^{t-u}}\otimes_{\tR}R$};
\node(20)[node distance=2.5cm, below of=A]{$\frac{I^{u}}{JI^{u-1}}\otimes
\frac{I^{t+1-u}}{JI^{t-u}}$};
\node(11)[node distance=6.6cm, right of=10]{$\frac{\tI^{t+1}}{\tJ\tI^{t}}\otimes_{\tR}R$};
\node(label) [node distance=2.5cm, above of=10]{$(\hbox{Figure 1: definition of }m)$};
\node(21)[node distance=2.5cm, below of=11]{$\frac{I^{t+1}}{JI^{t}}$};
\node(01)[node distance=1.5cm, above of=11]{$\omega_{\tR/\tK}\otimes_{\tR}R$};
\node(-11)[node distance=1.5cm, above of=01]{$\omega_{R/K}$};
\draw[->] (10) to node  {$m(\tI,u,t)\otimes R$} (11);
\draw[->] (20) to node  {$m(I,u,t)$} (21);
\draw[->] (10) to node {$\cong$}  (20);
\draw[->] (11) to node {} (21);
\draw[->] (01) to node {$\cong$}  (11);
\draw[to reversed-to] (-11) to node {}  (01);
\draw[->, dashed, bend left =15] (20) to node[above=30pt, right=15pt] {$m$}  (-11);
\draw[->, bend left = 50] (-11) to node {$\phi'$}  (21);
\end{tikzpicture}
\end{center}
From the left-hand vertical isomorphism we see that the source of the map $m(\tI, u,t)\otimes_{\tR} R$ is annihilated by $K$. Hence its
image in $\tI^{t+1}/\tJ \tI^{t}\otimes_{\tR}R\cong \omega_{\tR/\tK}\otimes_{\tR}R$ is contained in
$\omega_{R/K}$, yielding a map $m$ indicated by the dotted arrow in the diagram.

By our assumption on $v$ there exists $u$ with $(t+1)/2 - \varepsilon \leq u \leq (t+1)/2 + \varepsilon$, and then the
surjectivity of $m(I,u,t)$ implies that $\phi'$ is surjective. To prove that the surjection $\phi'$ is an isomorphism if $K$ is a geometric $s$-residual intersection,
it suffices to verify that the source and target of $\phi'$ are isomorphic locally at every associated prime $P$ of
the $R$-module $\omega_{R/K}$. But we have seen before that $K$ has codimension $s$, hence every such $P$ has codimension $s$, and
therefore cannot contain $I$. It follows that the source and target of $(\phi')_P$ are both isomorphic to the Gorenstein ring $(R/J)_P$.

To prove that $m$ is a perfect pairing, recall that $m(\tI,u,t)$ is a perfect pairing by Theorem~\ref{duality}.
According to Theorem~\ref{basic 1} and Proposition~\ref{new basic}(3), the module $\tI^{u}/\tJ \tI^{u-1}$ is a maximal Cohen-Macaulay module over the Cohen-Macaulay ring $\tR/\tK$. We proved above that $\underline x$ is a regular sequence on $\tR/\tK$, so it is also a regular sequence on $\tI^{u}/\tJ \tI^{u-1}$.
It follows that
$$
\Hom_{\tR}(\tI^{u}/\tJ \tI^{u-1}, \omega_{\tR/\tK}) \otimes_{\tR}R \cong
\Hom_{R}(\tI^{u}/\tJ \tI^{u-1}\otimes_{\tR}R, \omega_{\tR/\tK}\otimes_{\tR}R).
$$
The right hand module is isomorphic to
$$
\Hom_{R}(I^{u}/J I^{u-1}, \omega_{\tR/\tK}\otimes_{\tR}R),
$$
and because $I^{u}/J I^{u-1}$ is annihilated by $K$, this is isomorphic  to
$$
\Hom_{R}(I^{u}/J I^{u-1}, \omega_{R/K}).
$$
Since
\begin{align*}
\Hom_{\tR}(\tI^{u}/\tJ \tI^{u-1}, \omega_{\tR/\tK}) \otimes_{\tR}R &\cong
\tI^{t+1-u}/\tJ \tI^{t-u} \otimes_{\tR}R\\
&\cong
I^{t+1-u}/J I^{t-u} ,
 \end{align*}
there is a
composite isomorphism
$$
I^{t+1-u}/J I^{t-u}  \rTo^{\cong}  \Hom_{R}(I^{u}/J I^{u-1}, \omega_{R/K}) .
$$
The commutativity of the diagram in Figure 1 shows that this isomorphism
is induced by $m$, so we are done.
\end{proof}

\smallskip

\section{Examples and counterexamples on duality} \label{codim2 examples}

\subsection*{Residual intersections of codimension 2  ideals}

\begin{example}[{\bf Explicit duality}]\label{codim 2 perfect}
 Let $R$ be a local Gorenstein ring and $C$ an $(n+1) \times (n+s)$ matrix with entries in $R$, where $n\geq1$ and $s\geq2$.
Suppose that the maximal minors of $C$
generate an ideal $K$ of codimension $s$, the generic value. Set
$t := s-2$ and $M := \coker C$.
Buchsbaum and Eisenbud \cite{BE1}
(see also \cite[Appendix A.2.6]{E}) computed minimal free $R$-resolutions of the first $t+1$ symmetric powers of $M$, and observed that, for $0\leq u\leq t+1$, these are
perfect $R$-modules of codimension $s$, and that the resolutions of $\Sym_{u}(M)$ and
$\Sym_{t+1-u}(M)$ are dual to one another; that is,
$$
\Sym_{t+1-u}(M) \cong \Ext^{s}_{R}(\Sym_{u}(M),\, R) \cong \Hom_{R}(\Sym_{u}(M),\,\omega_{R/K}).
$$

If we assume that the entries of $C$ are in the maximal ideal and the residue field of $R$ is infinite, then, possibly after column 
operations, we may suppose the $(n+1)\times n$ submatrix $A$ consisting of the first $n$ columns of $C$ has the property that the $n\times n$ minors of $A$ generate an
ideal $I$ of codimension 2. (Reason: Since $K$ has codimension $s$,
we see that $N := \coker(C^*)$
is locally free of rank $s-1$ in codimension $<s$ in $R$. It follows from the theory of basic elements that after factoring out $s$ general generators of $N$ we obtain a module of codimension $\geq 2$. This is the module presented by $A^*$.)

In this situation, the ideal $I$ is  strongly Cohen-Macaulay. Huneke~\cite
{H2} showed that $K$ is an $s$-residual intersection of $I$, see also Theorem~\ref{codim2 det}.
In \cite{CNT} the duality statement above is generalized to residual intersections of any strongly Cohen-Macaulay ideal.

In addition, Andy Kustin and the second author observed  (unpublished) that for geometric residual intersections of codimension 2 perfect ideals, the  symmetric power 
$\Sym_{u}(I/J)$ is isomorphic to $I^{u}/JI^{u-1}$ in the range of $u$ that we consider, and we reprove this in
Theorem~\ref{codim2 det} below. This gives a concrete example of our theory. 

Let $B$ be the $(n+1)\times s$ matrix made from the last $s$ columns of $C$, so that
$C =
\begin{pmatrix}
 {A|B}
\end{pmatrix}.
$
Let $J$ be the image of the composite map
$$
R^{s}\rTo^{B} R^{n+1} \cong \bigwedge^{n}R^{n+1*}\rTo^{\bigwedge^{n}A^{*}} \bigwedge^{n}R^{n*} \cong R \, .
$$
By the Hilbert-Burch Theorem, $\coker A \cong \image  \bigwedge^{n}A^{*} \cong I$, and thus
$M= \coker
\begin{pmatrix}
{A|B}
\end{pmatrix}
 \cong I/J$.

\begin{theorem}
\label{codim2 det}
With notation and hypotheses as above,  $K=J:I$ and  thus $K$ is
an $s$-residual intersection of $I$.
Let $\Ibar$ be the image of $I$ in the ring $\Rbar:=R/K$.
If $I+K$ has codimension $\geq s+1$ $($so that $K$ is a
geometric $s$-residual intersection of $I$$)$, then
$$
I^{u}/JI^{u-1} \cong \Ibar^{u} \cong \Sym_{u}(\coker C)
$$
for $0\leq u\leq t+1$ $($interpreting $I^u/JI^{u-1}$ as $\Rbar$ when $u=0$$)$. In particular, $I^{u}/JI^{u-1}$ and $I^{t+1-u}/JI^{t-u}$ have dual, finite free $R$-resolutions.
\end{theorem}

Note that this does
not require the condition $G_{s}$.

\begin{proof} By assumption the codimension of the ideal $K$ of $(n+1)\times (n+1)$ minors of $C$ is $s$, so by
 \cite{BE2}, $\ann (\coker C) = K$. But  $\ann (\coker C) = \ann(I/J) = J:I$.

There are natural surjections
$$
\Sym_{u}(I/J)\twoheadrightarrow  I^{u}/JI^{u-1}  \twoheadrightarrow  \Ibar^{u}.
$$
Recall that the determinantal ideal $K$ is perfect of codimension $s$. Thus, if $K$ is a geometric $s$-residual intersection, then $\Ibar^{u}$ has grade $\geq 1$, and both
$\Ibar^u$ and $\Sym_{u}(I/J)$ are locally free of rank 1 at the associated primes of $\Rbar$. Since  $\Sym_{u}(I/J)$ is a
maximal Cohen-Macaulay $\Rbar$-module, it is torsion free, and thus the two epimorphisms are isomorphisms.
\end{proof}

\end{example}

\smallskip

There are two kinds of hypotheses on the ideal $I$ in Theorem~\ref{general duality}: the condition $G_{g+v}$ on
$I$ itself and the existence of a good deformation $\tI$. We will show in Example~\ref{can't weaken G} that the first cannot be weakened and,
in Examples~\ref{mystery module bis} and \ref{can't drop deformation}, that the second cannot be dropped.
Here we write $w:=g+v$.

\begin{example}[{\bf A codimension 2 perfect ideal satisfying $G_{w-1}$ but not $G_{w}$}]\label{can't weaken G}
The following examples show that, even for licci ideals, the condition $G_{w}$ in Theorem~\ref{general duality} cannot be replaced
by the condition $G_{w-1}$. They are based on the construction explained in Example~\ref{codim 2 perfect}.

By the Hilbert-Burch Theorem, any perfect codimension 2 ideal $I$ with $n+1$ generators is the ideal of $n\times n$ minors
of an $(n+1) \times n$ matrix. Such ideals satisfy the deformation assumption: they are  specializations of the generic ideal of minors, which satisfies the condition $G_{s}$ for every $s$, and all their Koszul homology modules are Cohen-Macaulay (\cite{AH}). (These are the original examples of the licci ideals mentioned in the introduction.)

Let $2\leq w\leq s$, let $R$ be a power series ring $k\lbracket \seq x {s}\rbracket$, and let $M_{s}$ be the $s\times(2s-1)$ ``Macaulay matrix'', where the $i$-th principal diagonal entries are $x_{i}$ and the other
entries are 0 (we illustrate with the case $s=5$):
$$M_{s}=\begin{pmatrix}
x_{1}&x_{2}&x_{3}&x_{4}&x_{5}&0&0&0&0\\
0&x_{1}&x_{2}&x_{3}&x_{4}&x_{5}&0&0&0\\
0&0&x_{1}&x_{2}&x_{3}&x_{4}&x_{5}&0&0\\
0&0&0&x_{1}&x_{2}&x_{3}&x_{4}&x_{5}&0\\
0&0&0&0&x_{1}&x_{2}&x_{3}&x_{4}&x_{5}\\
\end{pmatrix}.
$$
We define the ideal $I_{s,w}$ to be the ideal of $(s-1)\times (s-1)$ minors of the matrix $N_{s,w}$ made from columns 2 through $s$ of $M_{s}$ by
replacing the entry $x_{w}$ of the $(s-w+1)$ row with 0; for example  $I_{5,3}$ is  the ideal of $4\times 4$ minors of
$$N_{5,3}=\begin{pmatrix}
x_{2}&x_{3}&x_{4}&x_{5}\\
x_{1}&x_{2}&x_{3}&x_{4}\\
0&x_{1}&x_{2}&0\\
0&0&x_{1}&x_{2}\\
0&0&0&x_{1}\\
\end{pmatrix}.
$$
It is easy to see that $I_{s,w}$ is a perfect ideal of codimension 2, and by computing the codimensions of the ideals of lower order minors of $N_{s,w}$ one sees that $I_{s,w}$ satisfies $G_{w-1}$ but not $G_{w}$.

We consider the cases $4\leq s\leq 7$, and we construct an $s$-residual intersection $K_{s,w} = J_{s,w}:I_{s,w}$ of $I_{s,w}$ as follows:

Let $M'_{s,w}$ be the matrix constructed from $M_{s}$ by replacing columns 2 up to $s$ with the matrix $N_{s,w}$, and adding the variable $x_{w}$ to the entries in the
$(s-w+1)$ row and both the first and $(2s-w+1)$ columns. Thus
$$
M'_{5,3}=\begin{pmatrix}
x_{1}&x_{2}&x_{3}&x_{4}&x_{5}&0&0&0&0\\
0&x_{1}&x_{2}&x_{3}&x_{4}&x_{5}&0&0&0\\
x_{3}&0&x_{1}&x_{2}&0&x_{4}&x_{5}&x_{3}&0\\
0&0&0&x_{1}&x_{2}&x_{3}&x_{4}&x_{5}&0\\
0&0&0&0&x_{1}&x_{2}&x_{3}&x_{4}&x_{5}\\
\end{pmatrix}.
$$
Macaulay2 computations show that for $s\leq 7$ and any $2\leq w\leq s$,
the ideal $K_{s,w}$ generated by the maximal minors of $M'_{s,w}$ has the generic codimension, $s$, and we conjecture that this is true in general.

Assuming that $K_{s,w}$ has codimension $s$,  we can use Theorem~\ref{codim2 det} to show that $K_{s,w}$ is an $s$-residual intersection of $I_{s,w}$. Explicitly,
let $P_{s,w}$ be the submatrix of $M'_{s,w}$ consisting of the columns not in $N_{s,w}$; for example if $s=5,w=3$ then
$$
P_{s,w}=\begin{pmatrix}
x_{1}&0&0&0&0\\
0&x_{5}&0&0&0\\
x_{3}&x_{4}&x_{5}&x_{3}&0\\
0&x_{3}&x_{4}&x_{5}&0\\
0&x_{2}&x_{3}&x_{4}&x_{5}\\
\end{pmatrix}.
$$
After rearranging the columns of $M'_{s,w}$ we may write $M'_{s,w} = \begin{pmatrix}
 {N_{s,w}|P_{s,w}}
 \end{pmatrix}$. Thus we may apply
Theorem~\ref{codim2 det} to conclude that, if we take $J_{s,w}$ to be the ideal generated by the entries of the matrix
obtained by multiplying $P_{s,w}$ by the row of signed maximal minors of $N_{s,w}$, we will have $K_{s,w} = J_{s,w}:I_{s,w}$,
an $s$-residual intersection of $I_{s,w}$. For example $J_{5,3}$ is generated by the entries of the row vector
$$
\begin{pmatrix}
\Delta_{1}& -\Delta_{2}& \Delta_{3}& -\Delta_{4}& \Delta_{5}
\end{pmatrix}
\cdot P_{5,3} \, ,
$$
where $\Delta_{i}$ is the determinant of the matrix obtained from $N_{5,3}$ by omitting the $i$-th row.

We now consider Theorem~\ref{general duality} in the cases of the ideals $I_{s,w}$ and $J_{s,w}$ with $4\leq s\leq 7$. We have
$g=2$, and we consider values of $v$ in the range specified in the theorem, so that
$w = g+v = 2+v$ lies in the range $(s+1)/2 \leq w\leq s$. As explained above, the ideal $I_{s,w}$ satisfies the deformation hypothesis of the
theorem, and satisfies $G_{g+v -1} = G_{w-1}$ but not $G_{g+v}$. This has the effect of adding 1 to the lower bound, and subtracting 1 from the upper bound,
for $u$ in Theorem~\ref{general duality}. Thus for the triples $(s,w,u)$ in the list
$$
(4,3,1), (5,3,2), (5,4,1), (6,4,2), (6,5,1), (7,4,3), (7,5,2), (7,6,1),
$$
the theorem does not guarantee duality. Of course the same goes for the ``dual'' triples $(s,w,s-1-u)$.

Computations in Macaulay2 
show that, indeed, duality does not hold in these cases. 
To check this, we 
compute resolutions of
$I^u/JI^{u-1}$ and $I^{s-1-u}/JI^{s-2-u}$. When 
the total Betti numbers in the minimal resolutions over $R$ of these two modules are not dual to one
another, the duality clearly does not hold. It turns out that 
this occurs in each case. (We note that in other cases, where these have
 the same graded Betti numbers as in the generic case, they must be
reductions from the generic case, and thus dual to one another.)

Consider, as an example, the case $(s,w,u) = (5,4,1)$:
According to Macaulay2, the Betti table of the minimal graded free resolution of $I^{u}/JI^{u-1} = I/J$ is
\begin{verbatim}
 total: 5 9 84 180 135 35
     4: 5 9  .   .   .  .
     5: . .  .   .   .  .
     6: . .  .   .   .  .
     7: . .  .   .   .  .
     8: . . 84 180 135 35
\end{verbatim}
while the Betti table of the minimal graded free resolution of $I^{s-1-u}/JI^{s-2-u} = I^{3}/JI^{2}$ is
\begin{verbatim}
total: 35 136 188 106 28 9
   12: 35 136 183  87  1 .
   13:  .   .   .   .  . .
   14:  .   .   .   .  . .
   15:  .   .   5  19 27 8
   16:  .   .   .   .  . 1
\end{verbatim}
By local duality, the dual,  $\Hom(I/J, \omega_{R/(J:I)})$ of $I/J$ is isomorphic up to a shift in grading, to $\Ext^{5}_{R}(I/J,R)$. From the first resolution we see that the presentation of this module (as a graded module or over the power series ring) has 35 generators and 135 relations, whereas from the second Betti table
we see that the minimal presentation of $I^{3}/JI^{2}$ has 35 generators and 136 relations; thus $I/J$ is not dual to $I^{3}/JI^{2}$.

\end{example}

\begin{example}[{\bf duality not given by multiplication}]\label{mystery module}
Let $R = k\lbracket x,y,z \rbracket \supset I = (x,y)^{2}$ where $k$ is an infinite field.
The pair $(R,I)$ admits a deformation $(\tR, \tI)$, where
$\tR = k\lbracket z_{1,1},\dots,z_{2,3}, z \rbracket $, the ideal $\tI$ is generated by the $2\times 2$ minors of the generic matrix
$$Z :=\begin{pmatrix}
z_{1,1} &z_{1,2} &z_{1,3} \\
z_{2,1} &z_{2,2} &z_{2,3} \\
\end{pmatrix},
$$
and the specialization $\tR \to R$ sends $Z$ to the matrix
$$\begin{pmatrix}
x&y&0 \\
0&x&y
\end{pmatrix}.
$$
If $J$ is generated by 3 sufficiently general homogeneous polynomials
of degree 3 in $I$, then the ideal $K = J:I$ is a 3-residual intersection, so by Theorem~\ref{Huneke} or Theorem~\ref{general duality},
$I/J$ is self-dual.

Computation shows that $K = (x,y,z)^{3}$. Thus $\omega_{R/K} =\Hom_{k}(R/K,k)$ has Hilbert function $6,3,1$. The surjection
 $\phi': \omega_{R/K} \to I^{2}/JI$ described in Theorem~\ref{general duality}(1) is, in this case, the dual of  the
 inclusion  $(x,y)(R/K) \hookrightarrow R/K$. Thus the Hilbert function of $I^{2}/JI$ is $5,2$. We see that, unlike in Theorem~\ref{duality}, 
 there is no injection $I^{2}/JI \to \omega_{R/K}$ because the socle of the first module is 2-dimensional.

We also claim that, unlike in the situation of Theorem~\ref{duality}, the self-duality map of $I/J$ is not given by multiplication. Indeed, there can be no perfect pairing $I/J\otimes I/J \to I^{2}/JI$ because the target
is annihilated by $(x,y,z)^{2}$ while $I/J$ is not.

By Theorem~\ref{general duality} there is a duality map $I/J\otimes I/J \to \omega_{R/K}$,
and the multiplication map $I/J\otimes I/J \to I^{2}/JI$ is the composite of this map
with the surjection $\phi': \omega_{R/K} \to I^{2}/JI$ described in the same Theorem.
Moreover, the duality map is a symmetric surjection, induced by the corresponding
duality map in the generic case. Thus $R/K \oplus I/J \oplus \omega_{R/K}$ is a commutative standard graded Gorenstein algebra
over $R/K$ and $R/K \oplus I/J \oplus I^{2}/JI$ is a proper homomorphic image.

It is shown in~\cite{CNT} that, for residual intersections of strongly Cohen-Macaulay ideals, such as the one in this example, the duality between symmetric powers is always induced by multiplication.
\end{example}

\subsection*{Residual intersections of codimension 3 ideals}
Even when $I$ itself satisfies the condition $G_{s}$, the conclusion
of  Theorem~\ref{general duality} may fail if $I$ does not have a deformation whose Koszul homology modules
are Cohen-Macaulay. 

\begin{example}[{\bf No surjection $\omega_{R/K}\twoheadrightarrow I^{2}/JI $}] \label{mystery module bis}
Let $R = k\lbracket x_{1}, \dots, x_{5}\rbracket$ where $k$ is an infinite field, and let $I$ be the ideal of $2\times 2$ minors of the matrix
$$
\begin{pmatrix}
x_{1} & x_{2} & x_{3} & x_{4} \\
x_{2} & x_{3} & x_{4} & x_{5}
\end{pmatrix}\,.
$$
If we take $J$ to be the ideal generated by 4 sufficiently general cubic forms in $I$, then by Theorem~\ref{duality}, the multiplication map
$I/J\otimes I/J \to I^{2}/JI$ is a perfect pairing. We claim that, unlike in the situation of Theorem~\ref{general duality}, there is \emph{no} surjection
$\omega_{R/K}\twoheadrightarrow I^{2}/JI $: computation shows that  $I^{2}/JI$ requires 20 generators, whereas $\omega_{R/K}$ requires only 16.
Of course by Theorem~\ref{duality}, there is a natural injection $I^{2}/JI \hookrightarrow \omega_{R/K}$ such that the composite pairing is also a perfect pairing. However, unlike the situation in \cite{CNT}, the multiplication  $I/J \otimes I/J \to \Sym_{2}(I/J)$ is not a perfect pairing.

 Could there be some ``mystery module'' $X$ and maps
$$
\omega_{R/K} \longleftarrow X \longrightarrow I^{2}/JI
$$
that explains both Examples~\ref{mystery module} and \ref{mystery module bis}?

\end{example}

\begin{example}[{\bf No perfect pairing}]
\label{can't drop deformation} 

Let $s=5$ and take $I$ to be the ideal of the nondegenerate rational quartic curve in $\P^{4}$
or of the Veronese surface in $\P^{5}$ that is, the ideal of $2\times 2$ minors of either
$$
\begin{pmatrix}
 x_{0}& x_{1}& x_{2}& x_{3}\\
  x_{1}& x_{2}& x_{3}& x_{4}
\end{pmatrix}
\hbox { or }
\begin{pmatrix}
 x_{0}& x_{1}& x_{2}\\
  x_{1}& x_{3}& x_{4}\\
  x_{2}&x_{4}&x_{5}
\end{pmatrix}.
$$
These ideals satisfy $G_{5}$ and admit a 5-residual intersection $K = J:I$ where $J$ is generated by 5 general cubic forms in $I$. For each of the 
two ideals
$I$ above,  all Koszul homology modules are Cohen-Macaulay except the first, and they satisfy the sliding depth condition for Koszul homology. Nevertheless, Macaulay2 computation
 shows that the modules
$I/J$ and $I^{2}/JI$ are not dual to one another.

Computation shows that there is no useful duality among the first three symmetric powers either: $\Sym_{3}(I/J) \not \cong \omega_{R/K} ,$
$$
\Hom(\Sym_{2}(I/J), \omega_{R/K}) \not \cong I/J,\quad\Hom(I/J, \omega_{R/K}) \not \cong \Sym_{2}(I/J),
$$
 and likewise for dualizing into $\Sym_{3}(I/J)$.
\end{example}

\vspace{-0.1cm}

\subsection*{Residual intersections of a codimension 5 ideal}
\begin{example}\label{Johnson} Let
$R=  k\lbracket x_{1}, \dots, x_{10}, y_{1},\dots, y_{5}\rbracket$, where $k$ is an infinite field, and let $I$ be the ideal generated by the $4\times 4$ Pfaffians of
the generic skew symmetric matrix
$$
M =
\begin{pmatrix}
 0 &x_{1}&x_{2}&x_{3}&x_{4}\\
 -x_{1} &0&x_{5}&x_{6}&x_{7}\\
 -x_{2} &-x_{5}&0&x_{8}&x_{9}\\
 -x_{3} &-x_{6}&-x_{8}&0&x_{10}\\
  -x_{4} &-x_{7}&-x_{9}&-x_{10}&0\\
\end{pmatrix}
$$
together with the entries of the vector
$$
\begin{pmatrix}
 y_{1}&y_{2}&y_{3}&y_{4}&y_{5}
\end{pmatrix}
M.
$$
This is a prime ideal of codimension 5, and is a complete
intersection locally on the punctured spectrum, so $I$ satisfied $G_{15}$. Mark Johnson found this ideal as an example where
each of $R/I, R/I^{2}$ and $R/I^{3}$ are Cohen-Macaulay (and thus of depth 10), while $R/I^{4}$ has depth 6.  The ideal $I$ thus satisfies the Strong Hypothesis with $s=7$, but not $s=8$.

Let $J\subset I$ be generated by 7 general quadrics in $I$, and let $K = J:I$.
As in the previous example,
computation shows that there is no useful duality among the first three symmetric powers: $\Sym_{3}(I/J) \not \cong \omega_{R/K} ,$
$$
\Hom(\Sym_{2}(I/J), \omega_{R/K}) \not \cong I/J,\quad\Hom(I/J, \omega_{R/K}) \not \cong \Sym_{2}(I/J),
$$
 and likewise for dualizing into $\Sym_{3}(I/J)$.

 However, by Theorems~\ref{duality} and \ref{basic 1}, the multiplication map does give a perfect pairing
$$
I/J \otimes I^{2}/JI \rTo I^{3}/JI^{2} \cong \omega_{R/K}.
$$
\end{example}
\bigbreak

\section{Complementary module and socle}\label{socle section}

We begin by reminding the reader of the classic description of the socle of a complete intersection of equicharacteristic 0.
Recall that if $k$ is a field and $R$ is a complete local $k$-algebra, then the \emph{K\"ahler different} $\fD_K(R/k)\subset R$ is the
0-th Fitting ideal of the universally finite module of differentials $\Omega_{R/k}$; for example, if
$R = k\lbracket x_{1},\dots,x_{d}\rbracket/(a_{1}, \dots, a_{d})$, then $\fD_K(R/k)$ is the ideal generated by the Jacobian determinant
$$
\Delta = \det
\begin{pmatrix}
\frac{ \partial a_1}{\partial x_1} &\dots & \frac{ \partial a_1}{\partial x_d}\\
\vdots&\ddots&\vdots\\
\frac{ \partial a_d}{\partial x_1} &\dots & \frac{ \partial a_d}{\partial x_d}\\
\end{pmatrix}.
$$

\medskip

\begin{theorem}\label{classic socle}
If $k$ is a field of characteristic 0 and $R$ is a complete local $k$-algebra, then $\fD_K(R/k)$ is nonzero if and only
if $R$ is a 0-dimensional complete intersection, and in this case $\fD_K(R/k)$ is the socle of $R$.
\end{theorem}

This result was proven by Scheja and Storch~\cite{SS2} (see also Kunz~\cite{K1}). The basic ideas are due to Tate~\cite[Appendix]{MR}.
For the reader's convenience we give the classic arguments in Appendix~\ref{Sclassic socle}.

Throughout this section we suppose that $R$ is a local Gorenstein ring of dimension $d$ with maximal ideal $\gm$,  that $I\subset R$ is an ideal of codimension $g$, and
$K = J:I$ is an $s$-residual intersection, and we set $t =s-g$. If $T$ is any ring we write $Q(T)$ for the total ring of quotients obtained
by inverting every nonzerodivisor in $T$.

We want to identify the socle of $\omega_{R/K}$ in the
case $\dim R/K = d-s = 0$. We will show that, under suitable hypotheses, the socle of $\omega_{R/K}\cong I^{t+1}/JI^{t}$
 is generated by the image of the
Jacobian determinant of  generators of $J$ (Theorems~\ref{socle is jacobian} and \ref{Jacobian containment}).

 We begin with a general result about the socle of the local cohomology module $\rH^{d-s}_{\gm}(R/JI^{t})$.

\begin{theorem}\label{simplesocle} If $I$ satisfies the Strong Hypothesis with respect to $s$,
then $\, \rH^{d-s}_{\gm}(R/JI^{t})$ has a simple socle and the natural map
 $$
\rH^{d-s}_{\gm}(\omega_{R/K})\cong \rH^{d-s}_{\gm}(I^{t+1}/JI^t) \rTo \, \rH^{d-s}_{\gm}(R/JI^{t})
 $$
 is injective. In particular, the two modules have the same socle.
\end{theorem}

\begin{proof} Recall that $I^{t+1}/JI^t \cong \omega_{R/K}$ has dimension $d-s$ by Theorem~\ref{basic 1}. 
Hence $\, \rH^{d-s}_{\gm}(I^{t+1}/JI^t)  \ne 0$ and has nonzero socle. This
module embeds into $\, \rH^{d-s}_{\gm}(R/JI^{t})$ because 
${\rm depth} \, R/I^{t+1} \geq d-s$. Thus it remains to show that the socle of $\, \rH^{d-s}_{\gm}(R/JI^{t})$
is simple.

If $t=0$ the result is the usual duality for complete intersections, so we assume that $t>0$. We may harmlessly suppose that $k$ is infinite and that the generators $a_{1},\dots, a_{s}$ of
$J$ are general. Set $J_{i} = (a_{1},\dots, a_{i})$ and $K_i = J_i:I$.
By Lemma~\ref{principal ideal} the ideal $J_{i}:I$ is a geometric $i$-residual intersection for $g\leq i \leq s-1$.

From Proposition \ref{new basic}(2) we have an exact sequence
$$
 0\rTo \frac{I^{t}}{J_{s-1}I^{t-1}} \rTo^{a_{s}} \frac{R}{J_{s-1}I^{t}} \rTo\frac{R}{J_{s}I^{t}} \rTo 0 \, .
$$
The module in the middle has depth at least $d-s+1$ according to Proposition
\ref{new basic}(4). Hence the long exact sequence of local cohomology
gives an embedding
$$
\rH^{d-s}_{\gm}(R/JI^{t}) \subset \rH^{d-s+1}_{\gm}(I^t/J_{s-1}I^{t-1}).
$$
Now the theorem follows because, by
Theorem~\ref{basic 1},
$$
I^t/J_{s-1}I^{t-1} \cong \omega_{R/K_{s-1}}
$$
and $R/K_{s-1}$ is Cohen-Macaulay, so
$\rH^{d-s+1}_{\gm}(I^t/J_{s-1}I^{t-1})$ has simple socle.
\end{proof}

If $k$ is a field of characteristic 0 and $T$ is a local finite-dimensional $k$-algebra, then the trace homomorphism $\Tr_{T/k}$ is nonzero and annihilates
the maximal ideal, since the maximal ideal consists of nilpotent elements. Thus $\Tr_{T/k}$ generates the socle of
$\omega_{T} = \Hom_{k}(T,k)$.

From this point on we will assume that $R = k\lbracket \seq x d \rbracket$ is a power
series ring in $d$ variables over a field $k$ of characteristic 0. To identify the socle of $I^{t+1}/JI^{t}$ with the Jacobian determinant $\Delta$ of a given set of generators of $J$
in the case $s=d$ (Theorem~\ref{socle is jacobian}), we begin by making explicit the composite isomorphism
$$
I^{t+1}/JI^{t}\cong \omega_{R/K}\cong \Hom_{k}(R/K, k).
$$
We show that if $\Delta$ is in $I^{t+1}$ then, under this isomorphism,
 $\Delta$ corresponds to the trace homomorphism $\Tr_{(R/K)/k}$. This is accomplished in Theorem~\ref{Dedekind complementary module}.
  In order to do this, we establish a  result about the Dedekind complementary module of $R/K$ that
  requires $R/K$ to be reduced, and holds for $s<d$.

\begin{lemma}\label{derivative lemma}
 Let $k$ be a field of characteristic 0, and let $R=k\lbracket \seq x d\rbracket$.
 Let $T = R/L$ be a reduced Cohen-Macaulay  factor ring of $R$
 with $\codim L = s-1$. Let
 $a_1,\dots, a_{s-1}$ be elements of $L$ that generate $L$ generically, and let
 $a_s\in \gm$ be a nonzerodivisor modulo $L$. Let $x_1,\dots, x_d$ be general variables of $R$. Set
\begin{align*}
A_{i} &= k\lbracket x_i,\dots, x_d \rbracket \\
A' &= k\lbracket a_s,x_{s+1},\dots, x_d \rbracket \\
\Delta_i &= \det
\begin{pmatrix}
\frac{ \partial a_1}{\partial x_1} &\dots & \frac{ \partial a_1}{\partial x_i}\\
\vdots&\ddots&\vdots\\
\frac{ \partial a_i}{\partial x_1} &\dots & \frac{ \partial a_i}{\partial x_i}
\end{pmatrix} \, .
\end{align*}

We have$:$
 \begin{enumerate}
\item $\Omega_{\Q(T)/A_{s+1}} := Q(T) \otimes_T \Omega_{T/A_{s+1}}$ is a free $Q(T)$-module 
of rank 1 generated by $dx_s$.
\end{enumerate}
For any differential form $df$ we write $df/dx_s$ for the ratio as
elements of $\Omega_{\Q(T)/A_{s+1}}$.
\begin{enumerate}
\setcounter{enumi} 1
\item
 $
 \Delta_s = \frac{d a_s}{dx_s} \Delta_{s-1}.
 $
 \item
 $
\fC(T/A_s) = \frac{d a_s}{dx_s}\; \fC(T/A').
$
\end{enumerate}
\end{lemma}

\begin{proof} The $Q(T)$-module
$\Omega_{Q(T)/A_{s+1}}$ is presented
by the transpose of the $(s-1)\times s$ matrix
$$
\Theta =
\begin{pmatrix}
\frac{ \partial a_1}{\partial x_1} &\dots & \frac{ \partial a_1}{\partial x_s}\\
\vdots&\ddots&\vdots\\
\frac{ \partial a_{s-1}}{\partial x_1} &\dots & \frac{ \partial a_{s-1}}{\partial x_s}
\end{pmatrix}\,.
$$
We write $\Theta_i$ for $(-1)^i$ times the determinant of the $(s-1)\times (s-1)$ submatrix of $\Theta$ omitting the $i$-th column. Note that $\Delta_{s-1} = (-1)^s\Theta_{s}$.

\smallbreak\noindent
(1) Because  $\seq x d$ are general, the ring $A_{s}= k\lbracket x_s,\dots, x_d \rbracket$ is a  Noether normalization
of the reduced, equidimensional ring $T$, so the $Q(T)$-module
$\Omega_{Q(T)/k} := Q(T)\otimes_T\Omega_{T/k}$ is  free of rank $d-s+1$ with basis $dx_s, \ldots,dx_d$.
Thus $\Omega_{Q(T)/A_{s+1}}$ is  free of rank 1 with basis $dx_s$ as claimed.

\smallbreak\noindent (2)
It follows that $\Theta$ has rank $s-1$. Moreover,  the vector
$$
\begin{pmatrix}
\frac{ dx_1}{dx_s} \\
\vdots\\
\frac{ dx_s}{dx_s}
\end{pmatrix}
$$
is in $\ker \Theta$.
Of course $\Theta$ also annihilates the vector
$$
\begin{pmatrix}
\Theta_1 \\
\vdots\\
\Theta_s
\end{pmatrix}\,,
$$
and because the entries of either vector generate the unit ideal in $Q(T)$, the two vectors
$$
\begin{pmatrix}
\Theta_1 \\
\vdots\\
\Theta_s
\end{pmatrix}
\hbox{\quad and \quad}
\begin{pmatrix}
\frac{ dx_1}{dx_s} \\
\vdots\\
\frac{ dx_{s}}{dx_s}\\
\end{pmatrix}
$$
are proportional. Since $dx_s/dx_s = 1$ we get
$$
\Theta_i = \Theta_s \frac{dx_i}{dx_s} 
$$
for $i= 1, \dots, s$.

By the chain rule,
$$
\frac{da_s}{dx_s} = \sum_{i=1}^s\frac{\partial a_s}{\partial x_i} \frac{dx_i}{dx_s} \, ,
$$
so
$$
 \Theta_s \frac{da_s}{dx_s} = \sum_{i=1}^s \frac{\partial a_s}{\partial x_i}  \Theta_i
$$
in
$Q(T)$.
Expanding $\Delta_s$ along the last row we get
\begin{align*}
\Delta_s = (-1)^s \sum_{i=1}^s \Theta_i \frac{\partial a_s}{\partial x_i}
&=  (-1)^s\Theta_s \frac{da_s}{dx_s}\\
&= \Delta_{s-1}\frac{da_s}{dx_s}
\end{align*}
as required.

 \smallbreak\noindent (3) Since $a_{s}\in \gm$ is regular on $T$ and $\seq x d$ are general, the ring
 $A'$ is another Noether normalization of $T$. By \cite[9.2]{K2}
 $$
 \fC(T/A_s) \, dx_s\wedge\cdots\wedge dx_d = \fC(T/A') \, da_s\wedge\cdots\wedge dx_d
 $$
 in
 $$
 \bigwedge^{d-s+1}\Omega_{Q(T)/k}.
 $$
 Since $\Omega_{Q(T)/k}$ is a free $Q(T)$-module with basis $dx_s, \dots, dx_d$,
 $$
 da_s\wedge dx_{s+1}\wedge\cdots\wedge dx_d = \frac{da_s}{dx_s} dx_s\wedge dx_{s+1}\wedge \cdots\wedge dx_d.
 $$
\end{proof}

The next theorem is one of our main results. It gives an explicit description of the complementary 
module of residual intersections.

\begin{theorem}\label{Dedekind complementary module}
 Let $k$ be a field of characteristic 0, and let $R = k\lbracket \seq x d\rbracket$. Let $I\subset R$ be an ideal satisfying the Strong Hypothesis
 for some $s<d$, and let
 $J:I$ be a geometric $s$-residual  intersection such that
 $\Rbar = R/(J:I)$ is reduced. Let $a_1,\dots,a_s$  be general elements in $J$. Let $x_1,\dots, x_d$ be general variables in $R$ and write $A = k\lbracket x_{s+1}, \dots, x_d \rbracket $. We have
 $$
 I^{t+1}\Rbar = \Delta \;  \fC(\Rbar/A),
 $$
 where $\Delta$ is the Jacobian determinant
 $$
\Delta = \det
\begin{pmatrix}
\frac{ \partial a_1}{\partial x_1} &\dots & \frac{ \partial a_1}{\partial x_s}\\
\vdots&\ddots&\vdots\\
\frac{ \partial a_s}{\partial x_1} &\dots & \frac{ \partial a_s}{\partial x_s}\\
\end{pmatrix}.
$$
\end{theorem}

\bigbreak


 \begin{proof}
Since $R$ is a domain we must have $g>0$.
For $i\geq g-1$ we set:
\begin{align*}
 J_i &:= (a_1,\dots,a_i)\subset I;\\
 R_i&:= R/(J_i:I);\\
 A_i&:=k\lbracket x_{i+1},\dots,x_d \rbracket ;\\
A'_i&:=k\lbracket a_{i+1},x_{i+2},\dots,x_d \rbracket ;\\
 \Delta_i &:= \det
\begin{pmatrix}
\frac{ \partial a_1}{\partial x_1} &\dots & \frac{ \partial a_1}{\partial x_i}\\
\vdots&\ddots&\vdots\\
\frac{ \partial a_i}{\partial x_1} &\dots & \frac{ \partial a_i}{\partial x_i}
\end{pmatrix}.
\end{align*}
If $i=g-1$ then $J_{i}:I = J_{i}$ is generated by the regular sequence $\seq a {i}$. If $i\geq g$ then by Lemma~\ref{principal ideal}, the ideal $J_i:I$ is a geometric $i$-residual intersection.
By Theorem~\ref{basic 1}, the ring $R_i$ is Cohen-Macaulay of dimension $d-i$.
It follows that the geometric $i$-residual intersection $J_i:I$ is generically generated by
$a_1, \ldots, a_i$. Moreover by Proposition~\ref{new basic}(1), the element $a_{i+1}$ is 
regular on $R_i$ for $i \leq s-1$. Proposition~\ref{red lemma} shows that the ring $R_i$ is reduced.
Finally, Theorem~\ref{basic 1} and Proposition~\ref{new basic}(5) give $\omega_{R_i} \cong I^{i-g+1}R_i$
for any $i$.

Since the $\seq x d$ are general, the ring
$A_i$ is a Noether normalization of $R_i$. Since
$a_{i+1}\in\gm$ is a nonzerodivisor on $R_i$, the ring
$A_i'$ is also a Noether normalization of $R_i$.

By induction on $i = g-1,\dots, s$ we prove that
$$
 I^{i-g+1}R_i = \Delta_i \; \fC(R_i/A_i).
$$
The case $i=s$ is the statement of the theorem.

If $i=g-1$ the assertion is that $R_i =\Delta_i \fC(R_i/A_i)$, or equivalently
that $\fC(R_i/A_i) = \Delta_i^{-1} R_i$. This is
classically known because
$R_i = R/(a_1, \dots, a_i)$ and $a_1,\dots, a_i$ is a regular sequence;
we will give a self-contained proof of this fact in Section 8, see Corollary~\ref{Dedekind complementary module3}.

Now take $i\geq g$ and assume the result is known for $i-1$. Consider the following diagram
that will be explained below:
\bigbreak
\noindent
\begin{tikzpicture}
[every node/.style={scale=.75},  auto]
\node(00) {$I^{i-g}R_{i-1}$};
\node(01)[node distance=3.5cm, right of=00]{$\Delta_{i-1}\fC(R_{i-1}/A_{i-1})$};
\node(02)[node distance=4cm, right of=01]{$\Delta_{i}\fC(R_{i-1}/A'_{i-1})$};
\node(03)[node distance=4.5cm, right of=02]{$\Delta_{i}\fC(R_{i-1}/A'_{i-1})\Tr_{R_{i-1}/A'_{i-1}}$};
\node(04)[node distance=1.5cm, below of=03]
{$\Delta_{i}\Hom_{A'_{i-1}}({R_{i-1},A'_{i-1}})$};
\node(14)[node distance=1.5cm, below of=04]
{$\Hom_{A'_{i-1}}({R_{i-1},A'_{i-1}})$};
\node(24)[node distance=1.5cm, below of=14]
{$\Hom_{A_{i}}({R_{i-1}/(a_i),A_{i}})$};
\node(34)[node distance=1.5cm, below of=24]
{$\Hom_{A_{i}}({R_{i},A_{i}})$};
\node(44)[node distance=1.5cm, below of=34]
{$\Delta_{i}\Hom_{A_{i}}({R_{i},A_{i}})$};
\node(20)[node distance=4.5cm, below of=00]
{$I^{i-g}R_{i-1}/a_iI^{i-g}R_{i-1}$};
\node(40)[node distance=3.0cm, below of=20]
{$I^{i-g+1}R_{i}$};
\node(42)[node distance=3.3cm, right of=40]
{$\Delta_i\fC(R_i/A_i)$};
\node(43)[node distance=4.0cm, right of=42]
{$\Delta_i\fC(R_i/A_i)\Tr_{R_i/A_i}$};

\draw[->] (02) to node {$\cong$}  (03);
\draw[->] (42) to node {$\cong$}  (43);
\draw[->] (14) to node [right=3pt]{$\cong$}  (04);
\draw[->] (34) to node {$\cong$}  (44);
\draw[->>] (00) to node [left=3pt]{mod $a_i$}  (20);
\draw[->>] (14) to node {mod $a_i$}  (24);
\draw[>->] (40) to node {}  (20);
\draw[>->] (34) to node [right=3pt] {$\varepsilon$}  (24);

\draw[double distance=2pt][-] (00) to node {ind.hyp.}  (01);
\draw[double distance=2pt][-] (01) to node {7.3}  (02);
\draw[double distance=2pt][-] (03) to node {}  (04);
\draw[double distance=2pt][-] (43) to node {}  (44);

\draw[->, dashed] (20) to node {$\phi$}  (24);
\draw[->, dashed] (40) to node {$\psi$}  (42);

\end{tikzpicture}
\bigbreak

By the induction hypothesis and Lemma~\ref{derivative lemma}
we have
$$
I^{i-g}R_{i-1} = \Delta_{i-1}\fC(R_{i-1}/A_{i-1}) = \Delta_i\fC(R_{i-1}/A'_{i-1}).
$$
By Proposition~\ref{new basic}(1) the ideal $I^{i-g}R_{i-1}$ has positive grade,
so $\Delta_i$ is a nonzerodivisor
in $R_{i-1}$. The arrow marked mod $a_{i}$ on the right in the 
diagram is surjective because $R_{i-1}$ is a free $A'_{i-1}$-module.

The isomorphism $\phi$ is induced by the first row.
In the inclusion of
$I^{i-g+1}R_i \cong \omega_{R_i}$
in
$I^{i-g}R_{i-1}/a_iI^{i-g}R_{i-1} \cong\omega_{R_{i-1}/(a_i)}$
the first module is the annihilator of $L := \ker(R_{i-1}/(a_i) \twoheadrightarrow R_i)$;
see Corollary~\ref{reduction lemma}.
Similarly, we take
$$
\varepsilon: \Hom_{A_i}(R_i, A_i)\hookrightarrow
\Hom_{A_i}(R_{i-1}/(a_i), A_i)
$$
to be the map
induced by the surjection $R_{i-1}/(a_i) \twoheadrightarrow R_i$, so
the source of $\varepsilon$ is the annihilator of $L$ in the target of $\varepsilon$.
Since $R_i$ is generically
a finite separable extension of $A_i$ defined by the
vanishing of $a_1,\dots, a_i$, the element $\Delta_i$ is a nonzerodivisor of $R_i$.
Thus $\phi$ induces an isomorphism $\psi$ in the diagram.

We next will show that, regarded as a map of subsets of $Q(R_i)$, the map $\psi$ is the identity.
The source of $\psi$ contains a nonzerodivisor
by Proposition~\ref{new basic}(1). We may write it as
 the image of an element $v\in I^{i-g+1}$.
Since $\Delta_{i}$ is a nonzerodivisor on $R_{i}$, both source and target of $\psi$ are
fractional ideals containing nonzerodivisors, so $\psi$ is multiplication by
some element in $Q(R_i)$.
 To show that $\psi$ is the identity it suffices to show that $\psi(u) = u$ for some
nonzerodivisor $u\in I^{i-g+1}R_i$. We take $u$ to be the image of $\Delta_i v$ in $I^{i-g+1}R_i$.

Recall that $L\subset R_{i-1}/(a_i)$.
Since $I^{i-g+1}L \subset IL = 0$  we have $vL = 0$.
Since $L = \ker(R_{i-1}/(a_i) \twoheadrightarrow R_i)$ and both $R_{i-1}/(a_i)$ and $R_i$
are Cohen-Macaulay rings with Noether normalization $A_i$, it
follows that they are free $A_i$-modules, and thus
$ R_{i-1}/(a_i) \cong R_i\oplus L$ as $A_i$-modules.

From $vL=0$ one sees that
$$
\varepsilon(v\Tr_{R_i/A_i}) = v \Tr_{(R_{i-1}/(a_i))/A_i}.
$$
Following the maps in the diagram, we now see that
$\psi(\Delta_i v) = \Delta_i v$ as required.
\end{proof}

\begin{example} The following example illustrates a subtlety in the inductive proof above.
The conclusion of Theorem~\ref{Dedekind complementary module} shows that the image of $\Delta_s:=\Delta$ in $R_s:=\overline R$ is contained in
$I^{t+1}R_s$. The following example shows that $\Delta_s$ itself
may not be contained in $I^{t+1}$, and, moreover,
the image of $\Delta_s$ in $I^{t+1}R_s$ is not necessarily mapped
under the inclusion
$$
I^{t+1}R_s \hookrightarrow I^{t}R_{s-1}/a_s I^{t}R_{s-1}
$$
to the image of $\Delta_s$ in the target.

Take $s=2$ and let
 \begin{align*}
 R&=k\lbracket x,y,z\rbracket;\\
 I &=(z-x-y);\\
 J &= (a_1, a_2), \hbox{ where} \\
 a_1 &= xz-x^2-xy,\\
 a_2&=yz-yx-y^2.
 \end{align*}
We have
 \begin{align*}
 (a_1):I &= (x)\\
 J:I &= (x,y) = K.
\end{align*}
Computation shows that $\Delta_2\notin I^2$, and
the map of canonical modules
$I^2R_2 \hookrightarrow IR_1/a_2IR_1$
does not send the image of $\Delta_2$ to the image of $\Delta_2$.
\end{example}

\smallskip

From Theorem~\ref{Dedekind complementary module} we derive a formula for the Dedekind
complementary module of certain determinantal rings:

\begin{corollary}\label{determinantal rings}
 Let $k$ be a field of characteristic 0, and let $R = k\lbracket \seq x d\rbracket$. Let $C$ be an $(n+1) \times (n+s)$ matrix with entries in the maximal ideal of $R$, where $n \geq 1$ and $s \geq 2$, and assume that the maximal minors of $C$ generate an ideal $K$ of height  $s$, the generic value.
 Suppose that the ring $\Rbar=R/K$ is reduced. Let $D$ be an $(n+1) \times n$ matrix consisting of  $n$ columns of $C$, let $I$ be the ideal
 generated by the $n \times n$ minors of $D$, and let $a_1, \ldots, a_s$ be the $(n+1) \times (n+1)$ minors of $C$ that involve the $n$ columns of $D$. Let $x_1, \ldots, x_d$ be general variables
 in $R$, so that $\Rbar$ is module finite over $A = k\lbracket x_{s+1}, \dots, x_d \rbracket $. We have
 $$
 I^{s-1}\Rbar = \Delta \;  \fC(\Rbar/A),
 $$
 where $\Delta$ is the Jacobian determinant
 $$
\Delta = \det
\begin{pmatrix}
\frac{ \partial a_1}{\partial x_1} &\dots & \frac{ \partial a_1}{\partial x_s}\\
\vdots&\ddots&\vdots\\
\frac{ \partial a_s}{\partial x_1} &\dots & \frac{ \partial a_s}{\partial x_s}\\
\end{pmatrix}.
$$

Moreover, after suitable column operations on $C$, the submatrix $D$ may be chosen so that the ideal $I \Rbar$ has positive grade, and in this case $\Delta$ is a non-zerodivisor on $\Rbar$.
\end{corollary}

\begin{proof}
Let $\tilde C =(y_{i,j})$ be an $(n+1) \times (n+s)$ matrix of variables, and write $S= R\lbracket \{y_{i,j}\} \rbracket$  and $B=A \lbracket \{y_{i,j} \} \rbracket$.
Let $\tilde D$, $\tilde K$, $\tilde I$, ${\tilde a}_1, \ldots, {\tilde a}_s$, and $\tilde{\Delta}$ be the same objects as defined in the statement of the 
Corollary, using the matrix $\tilde C$ instead of $C$. Write $\Sbar=S/\tilde K$. Specializing $\tilde C$ to $C$, these objects specialize to the ones defined in the Corollary. The ideal $\tI$ is perfect of codimension 2 and satisfies the
Strong Hypothesis for $s$, $\tilde{a}_1, \ldots, \tilde{a}_s$ are generic elements of $\tilde I$, and ${\tilde K}=(\tilde{a}_1, \ldots, \tilde{a}_s): \tI$ is 
a geometric $s$-residual intersection of $\tI$, by Theorem~\ref{codim2 det} or \cite{H2}. 
 Theorem~\ref{Dedekind complementary module} and its proof show that
$$
{\tilde I}^{s-1}{\Sbar} \, {\rm Tr}_{\Sbar/B}=\tilde{\Delta} \, {\rm Hom}_{B}(\Sbar,B) \, .
$$
Since $\Sbar$ is a free $B$-module of finite rank, we have ${\rm Hom}_{B}(\Sbar,B) \otimes_{B}A ={\rm Hom}_{A}(\Rbar, A)$.
Taking images in this module, the equality above gives
$$I^{s-1}\Rbar \, {\rm Tr}_{\Rbar/A}={\Delta} \, {\rm Hom}_A(\Rbar,A) \, ,$$
and hence the main assertion of the Corollary.

Since $\Rbar$ is reduced, the ideal $K$ is generically a complete intersection, so 
the $n \times n$ minors of $C$ generate an ideal of positive grade in $\Rbar$. It follows that
after suitable column operations on $C$ we may choose the submatrix $D$ so that $I \Rbar$ has positive grade
in $\Rbar$. (Reason: the column space of  the matrix $\overline{C}$ over the ring $\Rbar$ has rank $n$, and thus the same is true for a general choice of $n$ columns.)
\end{proof}
 
 \smallskip
In the next results we apply our theory to certain 0-dimensional residual intersections. Our goal is to give  formulas for the  socles of their canonical modules as Jacobian determinants.

\begin{corollary}\label{socle is jacobian1}
 Let $k$ be a field of characteristic 0, and let $R= k\lbracket \seq x d\rbracket$. Let $I\subset R$ be an ideal satisfying the Standard Hypothesis with respect to $s=d$, let $J:I$ be a $d$-residual intersection, and
 set $t=d-g$. Let  $a_1,\dots,a_d$  be general elements in $J$, and let $\Delta$ be the Jacobian determinant of $a_1,\dots,a_d$.

 If $\, \Rbar= R/((a_1,\dots, a_{d-1}):I)$, then the image of $\Delta$ in $\Rbar$ is in $ I^{t}\Rbar$. Further,
the image of
$\Delta$  generates the socle of
$$
I^{t}\Rbar/(a_dI^{t}\Rbar)\cong \omega_{\Rbar/(a_d)} \, .
$$
\end{corollary}

\begin{proof} Notice that $(\seq a {d-1}):I$ is a geometric $(d-1)$-residual intersection
and $\Rbar$ is reduced, by Lemma~\ref{principal ideal} and Proposition~\ref{red lemma}. The module 
$I^t\Rbar/a_dI^t\Rbar$
is isomorphic to $\omega_{\Rbar/(a_d)}$ by Theorem~\ref{basic 1} and Proposition~\ref{new basic}, parts (5) and (1).

The hypothesis of the Corollary is sufficient to justify the  upper half of the diagram in the proof of 
Theorem~\ref{Dedekind complementary module} for the case $i=d$. The first row
of the diagram shows that the image of $\Delta = \Delta_d$ in $\Rbar$ 
lies in $I^{t}\Rbar$, and hence gives an element of $I^{t}\Rbar/a_dI^{t}\Rbar$.
The isomorphism $\phi$ maps this element to $\Tr_{(\Rbar/(a_d))/A_d}$, which generates
 the socle of $\Hom_{A_d}(\Rbar/(a_d), A_d)$.
\end{proof}

\vspace{-0.1cm}

\begin{theorem}\label{socle is jacobian}
 Let $k$ be a field of characteristic 0, and let $R = k\lbracket \seq x d\rbracket$. Let $I\subset R$ be an ideal satisfying the Standard Hypothesis with respect to $s=d$,  and let $J:I$ be a $d$-residual intersection. Let  $a_1,\dots,a_d$  be general elements in $J$, and let $\Delta$ be the Jacobian determinant of $a_1,\dots,a_d$.

There is an element $p\in (a_1,\dots, a_{d-1})$
such that $\Delta' := \Delta+p\in I^{t+1}$, and the image of $\Delta'$ generates the socle of
$$
I^{t+1}/JI^{t}\cong \omega_{R/(J:I)}.
$$

 Moreover, if
$\Delta\in I^{t+1}$ then the image of $\Delta$ generates this socle.
\end{theorem}

By Theorem~\ref{simplesocle}, the socles of $I^{t+1}/JI^{t}$ and $R/JI^t$ are the same, so Theorem~\ref{socle is jacobian} can also be interpreted as a result on the socle of $R/JI^t$.

\begin{proof} Recall that $I^{t+1}/JI^t \cong \omega_{R/(J:I)}$ by Theorem~\ref{basic 1}. 
Let $K_{d-1} = (\seq a {d-1}):I$ and write $\Rbar = R/K_{d-1}$. We will first prove that we
can take $p\in K_{d-1}$. By the first statement of Corollary~\ref{socle is jacobian1}, there is an element
$p_{1}\in K_{d-1}$ such that $\Delta+p_{1}\in I^{t}$.
 By
Corollary~\ref{reduction lemma} there is
a natural inclusion
$$
I^{t+1}/JI^{t} \hookrightarrow I^{t}\Rbar/(a_dI^{t}\Rbar).
$$
By the second statement of Corollary~\ref{socle is jacobian1}, the image of $\Delta+p_{1}$ generates the
socle of $I^{t}\Rbar/a_dI^{t}\Rbar$, and thus lies in the submodule $I^{t+1}/JI^{t}$ and
generates its socle.
In particular, there is an element $p_{2}\in K_{d-1}$ and an element $q\in a_{d}I^{t}$ so that
$\Delta+p_{1}+p_{2}+q\in I^{t+1}$, and the image of this element generates the socle of
$I^{t+1}/JI^{t}$. Since $q\in JI^{t}$ we may take $p = p_{1}+p_{2} \in K_{d-1}$.

By Theorem~\ref{classic socle}, $\Delta\in J$ if $t>0$, while $\Delta\in J:\gm$ if $t=0$, in which case $R/J$ is Gorenstein, and therefore in either case $\Delta\in I$. Thus $p\in I\cap K_{d-1}$. By Proposition~\ref{new basic}(5),
$p\in (a_{1},\dots, a_{d-1})$ as claimed.

If $\Delta\in I^{t+1}$ to begin with, we could take $p_{1}=p_{2}=0$ proving the last statement.
\end{proof}

In the graded case, Remark~\ref{graded case} identifies the socle up to homogeneous isomorphism,
$$
\soc \frac{I^{t+1}}{JI^{t}} \cong  
(\soc \omega_{R/(J:I)})(-\sum_{j=1}^{d} (\delta_{j}-1)) \cong k(-\sum_{j=1}^{d}(\delta_{j}-1)),
$$
so the socle has the same degree as the Jacobian determinant of $d$ homogeneous generators of $J$.

Motivated by Theorem~\ref{socle is jacobian}, we try to find conditions when $\Delta\in I^{t+1}$.

\begin{proposition}\label{JacInAPower}
Let $k$ be a perfect field,  let $R = k[x_{1},\dots,x_{d}]$ be a standard graded polynomial ring in $d$ variables, and let $J \subset R$ be an ideal. Set $e$ equal to the maximum of the codimensions of the minimal primes of $J$. If $J$ is generated by forms of the same degree $ >1$, then the $d\times d$ minors
of the Jacobian matrix of these forms are contained in the symbolic power $(\sqrt{J})^{(d-e+1)}$.
\end{proposition}

\begin{proof} Set $t = d-e$. If $t=0$ the result is trivial, so we
may assume $t>0$. Since $k$ is perfect we may assume that $k$ is algebraically closed. In this case $\sqrt J$ is
the intersection of the one-dimensional linear ideals that contain it. Inverting a linear form not in any minimal prime of $J$ and taking the degree 0 part, these become
maximal ideals. By Zariski's Main Lemma on Holomorphic Functions (see for example \cite[Corollary 1] {EH})  the $(t+1)$-st symbolic power of $\sqrt J$ in the dehomogenized ring is the intersection of the $(t+1)$-st powers of these maximal ideals, and thus in $R$
the ideal $(\sqrt{J})^{(t+1)}$ contains, hence is equal to, the intersection of the $(t+1)$-st powers of the 1-dimensional linear ideals that contain it.

Changing notation, it thus suffices to prove that if
$J$ is contained in the ideal $L = (x_{1},\dots,x_{d-1})$ and $f_{1},\dots, f_{d}$ are forms in  $J$ of degree $\delta>1$, then
$\det\jacobian(f_{1},\dots, f_{d})\in L^{t+1}$, where $\jacobian$ denotes the Jacobian matrix.
Write
$$
\begin{pmatrix} f_{1}\\ \vdots \\ f_{d}
\end{pmatrix}
= A
\begin{pmatrix} x_{1}\\ \vdots \\x_{d-1}
\end{pmatrix}
$$
for some $d\times d-1$ matrix $A$ with homogeneous entries of degree $\delta-1$.

We may write $A$ in the form $A = B+x_{d}^{\delta-1}C$, where  $B = (b_{i,j})$ has entries in $L$ and $C$ is a matrix of scalars. 
By the product rule,
\newcommand{\BigFig}[1]{\parbox{12pt}{\large #1}}

\begin{align*}
\jacobian(f_{1},\dots, f_{d}) &=
\sum_{j=1}^{d-1}x_{j}\jacobian(b_{1,j}, \dots, b_{d,j}) +
\begin{pmatrix}
&0\\
B
&\vdots\\
&0
\end{pmatrix}\\
&+(\delta-1)x_{d}^{\delta-2}
\sum_{j=1}^{d-1}x_{j}
\begin{pmatrix}
&c_{1,j}\\
0
&\vdots\\
&c_{d,j}
\end{pmatrix}
+x_{d}^{\delta-1}
\begin{pmatrix}
&0\\
C
&\vdots\\
&0
\end{pmatrix}.
\end{align*}

Let $D$ be the sum of the first two terms on the right hand side of this expression, and let $E$ be the sum of the
two remaining terms. These matrices have the following properties:

\smallskip
\noindent{\bf 1)} each column of $D$ has entries in $L$; and

\smallskip
\noindent{\bf 2)} the last column of $D$ has entries in $L^{2}$. This is because
the last column of the Jacobian matrix is
defined by differentiating with respect to $x_{d}$ .

\smallskip
\noindent On the other hand the rank of the scalar matrix
$C$ is at most
the codimension of $J$ localized at $L$, which is at most $e$. 
The last column of $E$ is a linear combination
of columns of $C$ with coefficients in $L$. Thus:

\smallskip
\noindent{\bf 3)} the rank of $E$ is at most $e \, $; and

\smallskip
\noindent{\bf 4)} the last column of $E$ has
entries in $L$.

\smallskip
\noindent These properties of $D$ and $E$ imply that $\det(D+E)\in L^{t+1}$ as required.
\end{proof}

\vspace{0.0 cm}

\begin{theorem}\label{Jacobian containment} Let $k$ be a field of characteristic 0, and let 
$R=k[x_1, \ldots,x_d]$. Let $I\subset R$ be a homogeneous ideal satisfying the Standard Hypothesis with $s=d$,
and let $J\subset I$ be an ideal generated by $d$ forms of a single degree $\delta >1$ such that
$J:I$ is a $d$-residual intersection.

If $I$ is  reduced and $\mu(I_{P})\leq \codim P -1 \,$ for all prime ideals $P\supset I$ with
$g<\codim P<d$, then the Jacobian determinant of any $d$ homogeneous generators of $J$ of degree $\delta$
is in $I^{t+1}$ and thus, by Theorem~$\ref{socle is jacobian}$, generates the socle
of $I^{t+1}/JI^{t} \cong \omega_{R/(J:I)}$.
\end{theorem}

\begin{proof} We may assume that $I \neq R$.
In this case $\sqrt J = I$. Because $I$ is Cohen-Macaulay, all minimal primes of $I$ and hence of $J$ have
the same codimension $g$.
By Proposition~\ref{JacInAPower}, the Jacobian determinant is contained in $I^{(t+1)}$.
By the assumption on the $\mu(I_{P})$, the powers and symbolic powers of $I$ coincide on the
punctured spectrum (\cite[4.9(d)]{U}). Therefore, $I^{(t+1)}/JI^{t}$ is contained in the finite length part of $R/JI^{t}$. 
The latter has a simple socle generated in the same degree $d(\delta-1)$ as the
Jacobian determinant by Theorem~\ref{simplesocle} and Remark~\ref{graded case}. Thus the image
of the Jacobian determinant lies in $\soc I^{(t+1)}/JI^{t} = \soc I^{t+1}/JI^{t}$. In particular, the Jacobian determinant
is in $I^{t+1}$.
\end{proof}

The theorem above can also be understood in terms of primary decompositions,
rather than residual intersections; in this formulation, the result is a natural generalization of Theorem \ref{classic socle}.

To explain this, let $R$ be a local Gorenstein ring of dimension $d$ and $J$
any ideal of codimension $g$ generated by $d$ elements. For our purpose we may assume that $R/J$ has depth zero. Consider a decomposition $J=I \cap L$,
where $L$ is the zero-dimensional primary component in any shortest primary
decomposition of $J$ and $I$ is the intersection of the primary components of
positive dimension. Notice that $L$ is contained in the ideal $K=J:I$, which
gives an embedding ${\omega}_{R/K} \hookrightarrow {\omega}_{R/L}$. Also observe that $K$ is a $d$-residual intersection of $I$.

Now assume that $I$ satisfies the Standard Hypothesis with $s=d$, set
$t= d-g$, and let $E \supset I^{t+1}/JI^t$ be an injective
envelope of $I^{t+1}/JI^t$ as a module over $R/L$. Since $I^{t+1}/JI^t$
is a canonical module of $R/K$, we may choose ${\omega}_{R/L}$ to be {\it equal} to $E$.

\begin{corollary}\label{Jacobian containment2} In addition to the
assumptions of the preceding two paragraphs suppose that $R=k\lbracket \seq x d\rbracket$ is a power series ring in $d$ variables over a field of characteristic
zero and that $J$ is generated by homogeneous polynomials $f_1, \ldots, f_d$ of a single degree $>1$.

If $I$ is  reduced and $\mu(I_{P})\leq \codim P -1$ for all prime ideals $P\supset I$ with
$g<\codim P<d$, then the socle of ${\omega}_{R/L}$ is generated by the
image in $I^{t+1}/JI^t$ of the Jacobian determinant of $f_1, \ldots, f_d$.
\end{corollary}

From examples it would seem that the formula for the socle as a Jacobian
holds without the reduced hypothesis and without the assumptions on the local numbers of generators
beyond the $G_{s}$ condition of our Standard Hypothesis. We can at least prove this for $g=1$.

\begin{proposition}\label{socle in principal case} Let $k$ be a field, and let $R=k[x_{1},\dots, x_{d}]$.
Let $I = (G) \subset R$ be a principal ideal generated by a nonzero form of degree $\gamma$ and let $F = f_{1}, \dots ,f_{d}$ be a
regular sequence of
forms of the same degree $\delta$. Assume that neither $\delta$ nor $\delta+\gamma$ is 0 in $k$, and let $J$ be the ideal generated by the sequence of forms
$GF$. The socle of $R/JI^{d-1}$, hence the socle of $I^d/JI^{d-1},$ is generated by the Jacobian determinant $\det \jacobian(GF)$.
\end{proposition}
\begin{proof}
By Theorem \ref{socle of Gorenstein} the socle of $R/(F)$ is generated by $\det \jacobian(F)$, so the socle of
$R/JI^{d-1}= R/(G^{d}F)$ is generated by $G^{d}\det \jacobian(F)$. By Lemma~\ref{jacformula}
this is
$$
\frac{\delta}{\delta+\gamma}\ \det \jacobian(GF).
$$
\end{proof}

\begin{lemma}\label{jacformula} Let $R=k[x_{1},\dots, x_{d}]$.
If $G$ is a form of degree $\gamma$ and $F = f_{1}, \dots, f_{d}$ is a sequence of
forms of the same degree $\delta$, then
$$
\delta\det \jacobian(GF) = (\delta+\gamma)\,G^{d}\det \jacobian(F).
$$
\end{lemma}
\begin{proof} Write $G_{j}$ for $\partial G/\partial x_{j}$ and
$f_{i,j}$ for $\partial f_{i}/\partial x_{j}$.  By the product rule,
$$
\jacobian(GF) = G \,  \jacobian(F) +
\begin{pmatrix}
f_{1}\\ \vdots\\ f_{d}
\end{pmatrix}
\begin{pmatrix}G_{1}&\cdots&G_{d}
\end{pmatrix}.
$$
The second summand has rank 1, so by the multilinearity of the determinant we have
$$
\det\jacobian(GF) = G^{d}\det \jacobian(F) +
G^{d-1}\sum_{i=1}^{d} \det J_{i},
$$
where $J_{i}$ is the matrix obtained from $\jacobian(F)$ by replacing the $i$-th row
by the row
$f_{i}\begin{pmatrix}G_{1}&\cdots&G_{d}
\end{pmatrix}.
$
Expansion along the first column shows that this sum is equal to
$$
\det
\begin{pmatrix}
G^{d}& G^{d-1}G_{1}&\cdots&G^{d-1}G_{d}\\
-f_{1}&f_{1,1}&\cdots&f_{1,d} \\
\vdots&\vdots&&\vdots\\
-f_{d}&f_{d,1}&\cdots&f_{d,d}
\end{pmatrix}.
$$
We multiply the first column by $\delta$, add $x_{i}$ times the $(i+1)$-st column to the
first column for all $i$, and use Euler's formula. From this we see that
\begin{align*}
\delta \det\jacobian(GF)  &=
\det
\begin{pmatrix}
(\delta+\gamma)G^{d}& G^{d-1}G_{1}&\cdots&G^{d-1}G_{d}\\
0&f_{1,1}&\cdots&f_{1,d} \\
\vdots&\vdots&&\vdots\\
0&f_{d,1}&\cdots&f_{d,d}
\end{pmatrix}\\ \\
&=
(\delta+\gamma)\, G^{d}\det\jacobian(F).
\end{align*}
\end{proof}

\begin{example}\label{counter-example1}
If we do not assume the forms generating $J$ have the same degree,
then the Jacobian need not be well-defined modulo $JI^t$, and
in particular its image may not generate the socle, as the following
example shows. Let $k$ be a field of characteristic $\neq 2,3$, and let $R=k[x,y]$. Let
$F$ be the regular sequence $x^{2}+y^{2}, x+y$ and set $I=(G)$ with $G=x$,  and $J = (GF)$. We have
$$JI: \det \jacobian (GF) = (G^2F):\det \jacobian (GF) = (x),
$$
so $\det \jacobian (GF)$ is not in the socle modulo $JI$. Moreover, $\det \jacobian (GF)$ is not
even contained in $I^2$.  However, we can replace $F$ by a different sequence of generators
$F' = x^{2}-xy, x+y$ for $(F)$, and then
 the Jacobian determinant of $GF'$ \emph{does} generate
the socle modulo $JI$.
\end{example}

\begin{example}\label{counter-example2} Over a field of characteristic 0, the polynomial
$$
f = (x^2-z)(xz-y^2)
$$
is the product of two of the quasihomogeneous generators of the
ideal of the space curve $C$ with parametrization $t \mapsto (t^2, t^3, t^4)$.
 The Jacobian ideal $J$ of $f$
has codimension 2. The scheme defined by $J$ has an
isolated singularity, so $J$ is generically reduced,
and thus also its unmixed part $I$ is
reduced. In fact, $I= (x^2-z, xz-y^2)$ is a prime complete intersection.

Nevertheless, one can compute that the Hessian determinant of $f$ is not even contained in $I^2$.
Thus
$f$ violates conjecture (3) of van Straten and Warmt~\cite[7.1]{SW}.

\vspace{0.25cm}

In the case when $R$ is regular local and $s=g = d$ (so $t=0$), there is another famous (and easier) formula for the socle
of $R/J$ -- it is generated by the image of the determinant of any ``transition'' matrix expressing the generators of $J$ as linear combinations of the generators
of the maximal ideal of $R$.
The following examples show that in Proposition~\ref{socle in principal case} with $t>0$ we cannot replace the Jacobian by such a transition matrix: the determinant could be outside the ideal $JI^{t}:\mm$ and could also be in $J I^{t}$ (it could even be 0). 
\end{example}

\begin{example} a) Let $R=k\lbracket x_{1},\dots,x_{d}\rbracket$, let $G\in (x_{1},\dots,x_{d})$ be nonzero and let $F_{1},\dots F_{d}$ be a regular sequence in $R$. Writing
$G = \sum_{i}a_{i}x_{i}$ we see that
$$
\begin{pmatrix}
GF_{1}\\
\vdots\\
GF_{d}
\end{pmatrix}
=
\begin{pmatrix}
F_{1}\\
\vdots\\
F_{d}
\end{pmatrix}
\begin{pmatrix}
 a_{1}&\cdots&a_{d}
\end{pmatrix}
\begin{pmatrix}
x_{1}\\
\vdots\\
x_{d}
\end{pmatrix}.
$$
We may take the rank 1 matrix
$$
A:=\begin{pmatrix}
F_{1}\\
\vdots\\
F_{d}
\end{pmatrix}
\begin{pmatrix}
 a_{1}&\cdots&a_{d}
\end{pmatrix}
$$
as transition matrix, and we have $\det A = 0$ as soon as $d\geq 2$.

\smallbreak
\noindent b) Let $R=k[x,y]$, where $k$ is a field of characteristic $\neq3$, and take $I= (G)$ with $G=x^{2}+y^{2}$, and $J=(GF)$ with $F = x,y$. If we replace the Jacobian matrix
$\jacobian(GF)$ by
$$
A :=\frac{1}{3} \ \jacobian(GF) +
\begin{pmatrix}
 -y&x\\
 -y&x
\end{pmatrix},
$$
then
$$
A
\begin{pmatrix}
 x\\y
\end{pmatrix} =
\frac{1}{3} \ \jacobian(GF) 
\begin{pmatrix}
 x\\y
\end{pmatrix} = 
\begin{pmatrix}
 GF_{1}\\GF_{2}
\end{pmatrix}
$$
and
$\det A$ is in $I$, but $\det A$ is not in the socle of $I/JI$.

\smallbreak
\noindent c) If in example b) we change $G$ to $xy$ leaving everything else the same, then $\det A$ is not even in $I$.

\smallbreak
\noindent d) If $F_{1}, F_{2}$ is a regular sequence of forms of degree 2 in $k[x,y]$ and $G=a_1x+a_2y$ is a non-zero form, then there 
are examples with $\det A \neq 0$ but
$\det A  \in JI = G^{2}(F_{1}, F_{2})$. For instance, take
$$
A =
\begin{pmatrix}
 a_{1}F_{1}-yG& a_{2}F_{1}+xG\\
 a_{1}F_{2}-yG & a_{2}F_{2}+xG
\end{pmatrix};
$$
the determinant in this case is $G^{2}(F_{1}-F_{2})$.
\end{example}

\medskip

\section{Appendix: Differents and socles for Gorenstein rings}
\label{Sclassic socle}

In this section we provide self-contained expositions of the classical results on differents and socles that we have used, mostly for complete intersections in characteristic 0.
More generally than is usually stated, these yield a formula for the socle of a zero-dimensional Gorenstein ring.
The results of this section are known, some in greater generality, but not easily available.
Classic references are by Noether \cite{N}, Berger \cite{B}, Tate \cite[Appendix]{MR}, Scheja and Storch \cite{SS1, SS2}, Kunz \cite{K1, K2}.

Let $A$ be a Noetherian ring, let $R$ be an $A$-algebra that is essentially of finite type, and write $R^e = R\otimes_AR$. Let $\bD$ be the kernel of the multiplication map
$\mu: R^e\to R$, so that we have an exact sequence
$$
0\rTo \bD \rTo R^e\rTo^\mu R\rTo 0 \,.
$$
We want to compare three measures of ramification:

\begin{itemize}
\item The
 \emph{K\"ahler different} $\fD_K(R/A)$, introduced in a different case in Section~\ref{socle section}, is defined to be ${\rm Fitt}_0^R(\Omega_{R/A})$.

 \item The \emph{Noether different} $\fD_N(R/A)$ is defined to be $\mu(\ann_{R^e}\bD)$.

\item The \emph{Dedekind different}
$\fD_D(R/A)$ is defined, for instance, when $A\subset R$ is a ring extension,
$A$ is a Noetherian normal domain,  $R$ is reduced and a
finitely generated torsion free $A$-module,
and $R/A$ is separable.
The complementary module $\fC(R/A)$ is the fractional $R$-ideal such that
$$
\Hom_A(R,A) = \fC(R/A)\, \Trace_{L/K},
$$
where $K=Q(A)$ and $L=Q(R)$ are the total rings of quotients of $A$ and $R$ respectively. The
Dedekind different is defined to be the inverse of the complementary module,
$\fD_D(R/A) = \fC(R/A)^{-1}$.
\end{itemize}

Because $\Omega_{R/A} \cong \bD \, \otimes_{R^e} R$ and
 ${\rm Fitt}_0^{R^e}(\bD)\subset \ann_{R^e} \bD$, it follows that
 $\fD_K(R/A)\subset \fD_N(R/A)$.

 The Dedekind different is an ideal because $A$ is normal. We also have $\fD_N(R/A) \subset \fD_D(R/A)$, 
 which implies that $\fD_N(R/A)\Hom_A(R,A)\subset R\;\Trace_{R/A}$. For a short proof see \cite[Formula (3.3) proved in Lemma 3.4]{LS}.
 The last containment can be an equality even when the Dedekind different is not defined:

\begin{theorem}\label{Noether}
Let $A$ be a Noetherian ring and let $R$ be an $A$-algebra that is finitely generated and free as an $A$-module. If $\, \Hom_A(R,A)$ is cyclic as an $R$-module, then
 $$
 \fD_N(R/A)\Hom_A(R,A) = R\; \Trace_{R/A}.
 $$
\end{theorem}

\begin{proof}
We will divide the proof into several parts:

 \noindent{\bf 1)}
 Because $R$ is a free $A$-module, the natural map
$$
\Phi: R\otimes_A R \rTo \Hom_{A}(\Hom_{A}(R,A),R)
$$
given by
$$
 s\otimes t\mapsto \bigl(\varphi\mapsto \varphi(s)t\bigr) \,
$$
is an isomorphism of $R-R$ bimodules. 
The annihilator of $\bD$ is the unique largest
$R-R$ submodule of $R\otimes_AR$ on which the left and the right $R$-module 
structures coincide, and the subset $\Hom_{R}(\Hom_{A}(R,A),R)$ has the same property in $\Hom_{A}(\Hom_{A}(R,A),R) $.
It follows that
$\Phi$ carries the annihilator of $\bD$
onto $\Hom_{R}(\Hom_{A}(R,A),R)$.

Since $R$ is a finitely generated free $A$-module and $\Hom_A(R,A)$ is cyclic as an $R$-module, we have $\Hom_A(R,A) \cong R$.
It follows that $\ann_{R^e} \bD$ is cyclic as an $R$-module.

\medbreak

\noindent{\bf 2)}
Let $\Gamma$ be a generator of $\ann_{R^e}\bD$.
Since $\Phi(\Gamma)$ generates
$$
\Hom_R(\Hom_A(R,A),R)$$
 and $\Hom_A(R,A)\cong R$ we see that $\Phi(\Gamma)$ is an $R$-isomorphism. Let
$\sigma = \Phi(\Gamma)^{-1}(1) \in \Hom_{A}(R,A).$
It follows that $\sigma \mu:R\otimes_AR\to A$ is
a symmetric, nonsingular $A$-bilinear form.
\medbreak

\noindent{\bf 3)}
Let $\{v_{i}\}$ be an $A$-basis of $R$, and  suppose that $\Gamma=\sum_{i} v'_{i}\otimes v_{i}$.
 We claim that $\sigma(v'_{i}v_{j}) = \delta_{i,j}$ -- that is, $\{v'_{i}\}$ is the dual basis of $\{v_{i}\}$ with respect to $\sigma\mu$.

Indeed, since $\Phi(\Gamma)$ is $R$-linear, we have $\Phi(\Gamma)(r\sigma) = r$ for every $r\in R$. Thus, for each $j$,
$$
v_j = \Phi(\Gamma)(v_j\sigma) =
\Phi(\sum_{i}v'_{i}\otimes v_{i})(v_j\sigma)  =
\sum_{i}(v_j\sigma)(v'_{i})v_{i} =
\sum_{i}\sigma(v_{j}v'_i)v_{i}.
$$
Since the $v_{i}$ form an $A$-basis, we see that $\sigma(v'_{i}v_{j}) = \delta_{i,j}$ as required.

\medbreak

\noindent{\bf 4)} Finally, we claim that $\Trace_{R/A} = \mu(\Gamma) \, \sigma$.
 Let $r$ be an element of $R$, regarded as an $A$-endomorphism of $R$ by multiplication. We have
$$
\mu(\Gamma)\sigma(r) = \sigma(\mu(\Gamma) r) = \sigma(\sum_{i}v_{i}'rv_{i}).
$$
Since $\{v'_i\}$ and $\{v_i\}$ are dual bases with respect to $\sigma\mu$, this
sum is equal to $\Tr_{R/A}(r)$.

Since $\fD_N(R/A) = \mu(\ann_{R^e} \bD) = R \, \mu(\Gamma)$, we see that
$$
 \fD_N(R/A) \, \Hom_A(R,A) = \fD_N(R/A) \, \sigma =  R\, \mu(\Gamma)\,  \sigma = R\; \Trace_{R/A}\,,
$$
 as required.
\end{proof}

\vspace{0.0cm}

\begin{theorem}\label{Dedekind complementary module2}
In addition to the assumptions in the definition of the Dedekind different above, suppose that $A$ is a regular local ring.
If $R$ is Gorenstein, then $\fD_D(R/A) = \fD_N(R/A)$.
\end{theorem}

\begin{proof}  We first verify that the assumptions of Theorem~\ref{Noether} are satisfied. Recall that $A \subset R$ 
and $R$ is a finitely generated $A$-module. For all maximal ideals $\m$ of $R$, the rings $R_{\m}$ have the 
same dimension as $A$, as can be seen for instance by tensoring with the completion of $A$, so that $R$ splits as
a product of local rings, and using the
torsion freeness of $R$ over $A$. Thus, since the rings $R_{\m}$ are Cohen-Macaulay, it follows that $R$ 
is a maximal Cohen-Macaulay $A$-module, hence a free $A$-module. Moreover, as the rings $R_{\m}$ are Gorenstein
and have the same dimension as $A$, the $R$-module ${\rm Hom}_A(R,A)$ is locally free of rank 1. 
Therefore ${\rm Hom}_A(R,A) \cong R$ because $R$ is semilocal.

Thus we may apply Theorem~\ref{Noether}. Since ${\rm Hom}_A(R,A)=\fC(R/A)\,\Trace_{R/A}$ by the definition of the
complementary module, the theorem shows that
$\fD_N(R/A) \, \fC(R/A) =R$, which gives $\fD_N(R/A)=\fC(R/A)^{-1}=\fD_D(R/A)$.
\end{proof}

\begin{theorem}\label{ci} Let $A \subset R$ be a ring extension, where
$A$ is regular local and $R$ is finitely generated and torsion free as an 
$A$-module. If $R$ is locally a complete intersection, then 
$$
R \cong A[\seq x n]_W/(F_1,\dots, F_{n})
$$
for 
some regular sequence $F_1,\dots, F_{n}$ of length $n$ in the polynomial ring $A[\seq x n]$
and some multiplicatively closed subset $W$. Write
$\Delta$ for the image in $R$ of the Jacobian determinant of
$F_1,\dots, F_{n}$ with respect to $\seq x n$.
One has 
$$\fD_N(R/A) = \fD_K(R/A) =R \, \Delta .$$
\end{theorem}

\begin{proof} Write $R \cong A[x_1, \ldots, x_n]/{\mathcal J}$. Since $R$ is a finitely generated torsion free $A$-module, it follows 
as in the previous proof that every maximal ideal $\m$ of $R$ has the same codimension $d:={\dim} \, A$. Since $\m$ must contain the maximal ideal of $A$,
its preimage $\mathfrak M$ in $A[x_1, \ldots, x_n]$ has codimension $d+n$. Hence the ideal ${\mathcal J}_{\mathfrak M}$ has codimension
$n$. Thus it is generated by $n$ elements because $R$ is locally a complete intersection. Write $W$ for the complement in $A[x_1, \ldots, x_n]$ of the
union of the finitely many maximal ideals $\mathfrak M$. By basic element theory, ${\mathcal J}_W$ is again generated by $n$ 
elements $F_1, \ldots, F_n$ that can be chosen to form a regular sequence in $A[x_1, \ldots, x_n]$.

To prove the claim about differents, first notice that $R$ is a Cohen-Macaulay ring and hence a free $A$-module, as shown in the previous proof.
As before, let $\bD$ be the kernel of the multiplication map $\mu: R^e = R\otimes_AR \to R$.
The preimage $\tilde \bD$ of $\bD$ in
$A[\seq x n]_W\otimes_{A} R$ is the kernel of the natural map to $R$,
so it is generated by a regular sequence $G=G_1, \ldots, G_n$ of length $n$.
The ideal
$\tilde \bD$ also contains the sequence $F\otimes 1:=F_1 \otimes 1, \ldots, F_n\otimes 1$, which is still a regular sequence because
$R$ is flat over $A$.

Notice that $\bD=\tilde{\bD}/(F\otimes 1)=(G)/(F\otimes 1)$.
The preimage in
\newline
$A[\seq x n]_W\otimes_{A} R \,$ of the
annihilator of $\bD$  may thus be written as 
$(F\otimes 1):(G)$. This ideal quotient is generated by $F\otimes 1$ and the determinant of
any matrix $\Theta$ expressing the elements of $F\otimes 1$ as  linear combinations of the elements of $G$, see \cite{Wi} or \cite{Bu}. 
It follows that $\fD_N(R/A)$ is generated by the image in $R$ of $\det \Theta$.

On the other hand, since $G$ is a regular sequence and since
$$
\bD\otimes_{R^e}R = \bD/\bD^2 \cong \Omega_{R/A},
$$
the image in $R$ of  $\Theta$ is a presentation matrix
of $\Omega_{R/A}$. Thus the image of $\det \Theta$ also generates the ideal
$\fD_K(R/A) = R \, \Delta$.
\end{proof}

\vspace{0.001cm}

\begin{corollary}\label{Dedekind complementary module3} If 
the assumptions in the definition of the Dedekind different and the hypotheses of Theorem~$\ref{ci}$ 
are satisfied, then
$$\fC(R/A) = \fD_K(R/A)^{-1}= R \, \Delta ^{-1} .$$
\end{corollary}
\begin{proof} One uses Theorem \ref{Dedekind complementary module2}, Theorem \ref{ci}, and the fact that the fractional ideal $\fC(R/A)$ is
invertible, hence reflexive.
\end{proof}

\vspace{0.0cm}

\begin{theorem}\label{socle of Gorenstein}
 If $R$ is a local Gorenstein algebra over a field $k$ with ${\rm dim}_kR$
 finite and not divisible by the characteristic of $k$, then
 ${\frak D}_N(R/k)$
 is equal to the socle of $R$. If, moreover, $R$ is a complete intersection, then the socle of $R$ is generated by the Jacobian determinant.
 \end{theorem}

\begin{proof}  Since the trace of
any nilpotent element is 0, the trace lies in the socle of $\Hom_{k}(R,k)$ and generates it if the characteristic of $k$ does
not divide $\dim_{k}R$. Thus Proposition \ref{Noether} implies that
$\fD_N(R/k) \, {\rm Hom}_k(R,k)$ is the socle of ${\rm Hom}_k(R,k)$.
Therefore $\fD_N(R/k)$ is the socle of $R$ since
${\rm Hom}_k(R,k)\cong R$ as $R$-modules.

Finally, if $R$ is a complete intersection, then
${\frak D}_N(R/k) = {\frak D}_K(R/k)$ is generated by the Jacobian determinant, by Theorem~\ref{ci}.
\end{proof}

\vspace{0.0cm}

\proof[Proof of of Theorem~\ref{classic socle}]
One implication is a special case of Theorem~\ref{socle of Gorenstein}. To prove the opposite implication
 we must show that the K\"ahler different $\fD_K(R/k)$ is
0 when $R = k\lbracket x_{1}, \dots, x_{d}\rbracket/(a_{1}, \dots, a_{s})$ is not a 0-dimensional  complete intersection.

First suppose $R$ is 0-dimensional and not a complete intersection. Replacing the $a_{i}$ by general linear combinations, we may assume that
any $d$ of the $a_{i}$ form a regular sequence. By the previous theorem, the Jacobian determinant
of $a_{i_{1}},\dots,a_{i_{d}}$ generates the socle modulo $(a_{i_{1}},\dots,a_{i_{d}})$ and is thus contained in
$(a_{1}, \dots, a_{s})$ as required.

\def\JJ{{\mathcal J}}
Finally, suppose that $R$ is not 0-dimensional. To simplify the notation, set  $\m = (x_{1}, \dots, x_{d})$ and $\JJ=(a_{1}, \dots, a_{s})$ and suppose that $s$ is the minimal number of generators of $\JJ$. We may assume that $R$ is not a complete intersection since
otherwise $\fD_K(R/k) = 0$. For any sufficiently large integer $n$, the Artin-Rees Lemma and the Principal Ideal Theorem together imply that $\JJ+\m^{n}$ requires at least $s+\dim \JJ$ generators. Thus, $R/\m^{n}R$ is not a complete intersection. 

We conclude from the 0-dimensional argument that, for any $n\gg 0$,
$\fD_K((R/\m^{n}R)/k) = 0$. In particular, $\fD_K(R/k)$ is in $\m^{n}R$. By the Krull Intersection Theorem,
$\fD_K(R/k) = 0$.
\qed

\medskip

\medskip

\bibliographystyle{alpha}

\bigskip

\vbox{\noindent Author Addresses:\par
\smallskip
\noindent{David Eisenbud}\par
\noindent{Mathematical Sciences Research Institute,
Berkeley, CA 94720, USA}\par
\noindent{de@msri.org}\par

\medskip
\noindent{Bernd Ulrich}\par
\noindent{Department of Mathematics, Purdue University, West Lafayette, IN 47907, USA}\par
\noindent{ulrich@math.purdue.edu}\par
}

\end{document}